\documentclass[12pt]{amsart}
\usepackage{graphicx,color}
\usepackage{enumitem}
\usepackage{amsmath}
\usepackage{amssymb, mathrsfs}
\usepackage{amsbsy}
\usepackage{amscd}
\usepackage{amsthm}
\usepackage{bm}
\usepackage{hyperref}
\usepackage{tikz,wrapfig}
\usepackage{caption}
\let\originalleft\left
\let\originalright\right
\renewcommand{\left}{\mathopen{}\mathclose\bgroup\originalleft}
\renewcommand{\right}{\aftergroup\egroup\originalright}

\newcommand{\N}{\mathbb{N}}
\newcommand{\R}{\mathbb{R}}
\newcommand{\Z}{\mathbb{Z}}
\newcommand{\E}{\mathbb{E}}
\newcommand{\sfH}{\mathsf{H}}
\newcommand{\sfR}{\mathsf{R}}
\newcommand{\sfD}{\mathsf{D}}
\newcommand{\sfT}{\mathsf{T}}
\newcommand{\dd}{\,{\text d}}

\newcommand{\set}[1]{\left\{#1\right\}}

\renewcommand{\P}{\mathbb{P}}

\newcommand{\calH}{\mathcal{H}}

\newtheorem{theorem}{Theorem}[section]
\newtheorem*{theorem*}{Theorem}
\newtheorem{lemma}[theorem]{Lemma}
\newtheorem{proposition}[theorem]{Proposition}
\newtheorem*{proposition*}{Proposition}

\newtheorem{thmx}{Theorem}

\theoremstyle{definition}
\newtheorem{remark}[theorem]{Remark}

\numberwithin{equation}{section}

\begin{document}

\title[Quenched tail estimate for RWRS and RCM II]
{Quenched tail estimate for the random walk in random scenery and in random layered conductance II}
\author{Jean-Dominique Deuschel}
\address[Jean-Dominique Deuschel]
{Institut f\"ur Mathematik, Technische Universit\"at Berlin, Berlin, Germany}
\email{deuschel@math.tu-berlin.de}
\author{Ryoki Fukushima}
\address[Ryoki Fukushima]
{Research Institute for Mathematical Sciences, 
Kyoto University, Kyoto, Japan}
\email{ryoki@kurims.kyoto-u.ac.jp}
\date{\today}

\begin{abstract}
This is a continuation of our earlier work~[Stochastic Processes and their Applications, 129(1), pp.102--128, 2019] on the random walk in random scenery and in random layered conductance. We complete the picture of upper deviation of the random walk in random scenery, and also prove a bound on lower deviation probability. Based on these results, we determine asymptotics of the return probability, a certain moderate deviation probability, and the Green function of the random walk in random layered conductance. 
\end{abstract}

\maketitle


\section{Introduction and main results}
This paper is a continuation of our earlier work~\cite{DF18}. In that paper, we studied upper deviation estimates for the random walk in random scenery in the case where the random scenery is non-negative and has a Pareto distribution. 
The results were applied to establish tail estimates for a random conductance model with a layered structure. 

There are two situations left incomplete in~\cite{DF18} for the random walk in random scenery. The first is the regime where the upper deviation probability exhibits a power law decay. Such a regime was proved to exist in $d\ge 3$ but the precise decay rate has not been obtained. The second is the lower deviation estimate. In this paper, we give the precise estimates in the first regime (Theorem~\ref{power-law}), that is needed for the asymptotic of the moderate deviation and the Green function for the random conductance model explained below, and a partial result for the second problem of lower deviation (Proposition~\ref{prop:RWRSlower}). 

As in the previous work, these results have applications to a random conductance model with a layered structure (see Section~\ref{sec:RCM-results} for the precise formulation). In fact, one of our main motivation is to show that in contrast with the standard situation, where the random conductance are independent and identically distributed (cf.~\cite{BD10}), the upper tail of distribution of the conductance in a layered structure yields an anomalous behavior of the heat kernel and the Green function. The first application is to the on-diagonal behavior of the heat kernel. We determine the so-called spectral dimension of the integer lattice weighted by the random conductance, and prove a sharp criterion for the recurrence and transience (Theorem~\ref{thm:on-diag}). The spectral dimension exhibits a non-standard behavior as soon as the first moment of the conductance is infinite. The second application is to a moderate deviation estimate. Similarly to the random walk in random scenery, we found a power law decay regime in~\cite{DF18}, and in this paper the sharp asymptotics is determined (Theorem~\ref{thm:RCM-power}). To the best of our knowledge, this slow decay of moderate deviation probability is a new phenomenon in the context of random walk in random environment. The last application is to derive estimates for the Green function (Theorem~\ref{thm:Green}). In contrast to the spectral dimension, the decay of the Green function exhibits a non-standard behavior even for finite mean conductance when $d\ge 5$. 

\subsection{Results for random walk in random scenery}
\label{sec:RWRS-results}
Let $(\{z(x)\}_{x\in\Z^d},\P)$ be a family of non-negative, independent and identically distributed random variables whose law satisfies
\begin{equation}
 \P(z(x)> r)= r^{-\alpha+o(1)} \textrm{ as }r\to\infty
\label{ass-tail}
\end{equation}
for some $\alpha>0$. The random walk in random scenery is the additive functional of the continuous time simple random walk $((S_t)_{t\ge 0},(P_x)_{x\in\Z^{d}})$ on $\Z^d$ defined as follows:
\begin{equation}
 A(t)=\int_0^t z(S_u)\dd u. 
\end{equation}
For the background and related works, see the introduction of~\cite{GPPdS14,DF18}. Among many results, the asymptotic behaviors of $A(t)$ in our setting are studied in~\cite{KS79,CGPP13}, which say that $A(t)$ scales like 
\begin{equation}
\label{eq:scaling}
t^{s(d,\alpha)+o(1)} \textrm{ with }s(d,\alpha)=
\begin{cases}
{\alpha+1 \over 2\alpha}\vee 1, &d=1,\\
{1 \over \alpha}\vee 1,&d\ge 2
\end{cases}
\end{equation}
under the product measure $\P\otimes P_0$. In fact, finer distributional limit results are established in~\cite{Bor79a,Bor79b,KS79,CGPP13}. 

In this paper, we study quenched tail estimates for $A(t)$, that is, conditionally on $z$. The following upper tail estimates are proved in the previous work~\cite{DF18}:
\begin{thmx}
[Theorem~1 in~\cite{DF18}]\label{RWRS}
Let $\rho>\frac{\alpha+1}{2\alpha}\vee 1$ for $d=1$ and $\rho>\frac{d}{2\alpha}\vee 1$ for $d\ge 2$. Then there exists $p(\alpha,\rho)>0$ such that $\P$-almost surely, 
\begin{equation}
P_0\left(A(t) \ge t^{\rho}\right)
= \exp\left\{-t^{p(\alpha,\rho)+o(1)}\right\}
\label{rwrs}
\end{equation}
as $t\to\infty$. 
Furthermore, when $d=1$ and $\rho< \frac{\alpha+1}{2\alpha}\vee 1$
or $d\ge 2$ and $\rho< \frac{d}{2\alpha}\vee 1$, 
$\P$-almost surely the above probability is bounded from 
below by a negative power of $t$. 
\end{thmx}
\begin{thmx}
 \label{LDP}
Let $d=1$ and $\alpha>1$ or $d\ge 2$ and $\alpha>\frac{d}{2}$. 
Then for any $c>\E[z(0)]$, $\P$-almost surely, 
\begin{equation}
P_0\left(A(t) \ge ct\right)
= 
\begin{cases}
\exp\left\{-t^{\frac{\alpha-1}{\alpha+1}+o(1)}\right\},&d=1,\\[5pt]
\exp\left\{-t^{\frac{2\alpha-d}{2\alpha+d}+o(1)}\right\},&d\ge 2 
\end{cases}
\label{ldp}
\end{equation}
as $t\to\infty$. 
\end{thmx}

In this paper, we provide sharp estimates in the power law decay regime of Theorem~\ref{RWRS} and an estimate for a lower deviation probability.

Let us start with the upper deviation in the case $d\ge 3$.
The reason why we have a power law decay in the regime $d\ge 3$ and $\rho<\frac{d}{2\alpha}\vee 1$ can be explained as follows: The random walk can reach any point inside $[-t^{1/2},t^{1/2}]^d$ with probability $ct^{-d/2}$. Since the highest value of $z$-field in this box is $t^{\frac{d}{2\alpha}+o(1)}$, the deviation up to this value can be achieved with probability at least $ct^{-d/2}$.
The first result in this paper provides a sharp estimate in this regime. 
\begin{theorem}
\label{power-law}
Let $d \ge 3$ and $\alpha<d/2$. Then for any $\rho\in(\tfrac{1}{\alpha}\vee 1, \tfrac{d}{2\alpha})$, 
\begin{equation}
 P_0(A(t) \ge t^\rho)=t^{-\alpha\rho+1+o(1)} \textrm{ as }t\to\infty
\label{eq:unpinned}
\end{equation}
and
\begin{equation}
 P_0(A(t) \ge t^\rho,S_t=0)=t^{-\alpha\rho+1-\frac{d}{2}+o(1)} \textrm{ as }t\to\infty.\label{eq:lbd_pinned}
\end{equation}
If $\alpha>1$ in addition, then the same bounds hold with $t^\rho$ replaced by $ct$ for any $c>\E[z(0)]$.
\end{theorem}
In the proof, we will see that these asymptotics are coming from a simple specific strategy: the random walk visits the intermediate level set $\{x\in\Z^d\colon z(x)\ge t^\rho\}$ and spends a unit time there. The right-hand side of~\eqref{eq:unpinned} turns out to be the probability for the random walk to visit that level set before time $t$.

In the remaining case where $d=1,2$ and the conditions in the last part of Theorem~\ref{RWRS} are satisfied, we expect from~\eqref{eq:scaling} that the left-hand side of~\eqref{rwrs} tends to one as $t\to\infty$. This is confirmed by the following lower tail estimate. It shows that, in contrast to the upper deviation, any small lower deviation causes a stretched exponential decay in all dimensions. 
\begin{proposition}
\label{prop:RWRSlower}
Let $d \ge 1$. For any $\epsilon>0$, there exists $c(\epsilon)>0$ such that $\P$-almost surely, 
\begin{equation}
 P_0\left(A(t)\le t^{s(d,\alpha)-\epsilon}\right)
\le \exp\{-t^{c(\epsilon)}\}
 \label{lower-deviation}
\end{equation}
for all sufficiently large $t$. Furthermore, if $\E[z(0)]<\infty$, then the same bound holds with $t^{s(d,\alpha)-\epsilon}$ replaced by $t(\E[z(0)]-\epsilon)$. 
\end{proposition}
It would be an interesting problem to determine the precise decay rate as a function of $\epsilon$. 
\subsection{Results for random walk in layered conductance}
\label{sec:RCM-results}
Let us define the conductance field on $\Z^{1+d}$ by
\begin{equation}
  \omega(x,x\pm \mathbf{e}_i)= 
 \begin{cases}
  z(x_2),& i=1,\\
  1,& i\ge 2,
 \end{cases}
\end{equation}
where we write $x=(x_1,x_2)\in\Z^{1+d}$ with $x_1\in\Z$ and $x_2\in \Z^d$.  
Then the random walk $((X_t)_{t\ge 0},(P^{\omega}_x)_{x\in\Z^{1+d}})$ is defined through its generator
\begin{equation}
\mathcal{L}^\omega f(x)=\sum_{|e|=1}\omega(x,x+e)
(f(x+e)-f(x)). 
\label{eq:def_RCM}
\end{equation}
This is the variable speed random walk in the random conductance field $\omega$. One of the primary interests in this type of model is the central limit behavior. When $\E[z(0)]<\infty$, weak convergence results, such as the quenched functional central limit theorem, are relatively easy to establish since the environment is \emph{balanced} and \emph{reversible} for the environment viewed from the particle process. A similar model in the discrete time is called ``toy model'' in~\cite[Section 2.3]{Bis11}. In this paper, we will focus on the local central limit theorem type results. Our result in fact covers the case $\E[z(0)]=\infty$, where the heat kernel exhibits anomalous behaviors. 

This model is related to the random walk in random scenery through the following representation. Let $((S^1_t,S^2_t)_{t\ge 0},(P_x)_{x\in\Z^{1+d}})$ be the continuous time simple random walk on $\Z^{1+d}$ which jumps to each of the neighboring sites at rate one, where $S^1$ is the first one-dimensional component and $S^2$ is the remaining $d$-dimensional component. 
Then the process $((X_t)_{t\ge 0},(P^{\omega}_x)_{x\in\Z^{1+d}})$ has the representation
\begin{equation}
 (X_t)_{{t\ge 0}}=(S^1_{A^2(t)}, S^2_t)_{{t\ge 0}},
\label{representationRWRS}
\end{equation}
where the clock process is defined by $A^2(t)=\int_0^t z(S^2_u)\dd u$. In particular, it follows that $X_t$ scales like $t^{\frac{s(d,\alpha)}{2}}$ in $\mathbf{e}_1$-direction and $t^{\frac12}$ in the other directions. 
If we define a distance between two points $x=(x_1,x_2)$ and $y=(y_1,y_2)$ in $\Z^{1+d}$ by
\begin{equation}
 \bar{d}(x,y)=|x_1-y_1|^{\frac{1}{s(d,\alpha)}}+|x_2-y_2|
\label{eq:intrinsic}
\end{equation}
with $|\cdot|$ denoting the Euclidean norm on $\Z^d$, then $X_t$ obeys the diffusive $t^{\frac12}$ scaling in this distance. We will see that this distance plays a similar role as the so-called \emph{intrinsic distance}; see~\cite{ADS19} for a recent study in the random conductance setting. 
The representation~\eqref{representationRWRS} plays a key role in the proofs of following results. 

Our first result concerns the behavior of the \emph{on-diagonal} heat kernel. Roughly speaking, its behavior is similar to the case $\omega\equiv 1$ when $\alpha\ge 1$, and different when $\alpha<1$. 


\begin{theorem}
\label{thm:on-diag}
If $\E[z(0)]=\infty$, then $\P$-almost surely,
\begin{equation}
P^\omega_0(X_t=0)=
\begin{cases}
t^{-{3\alpha+1 \over 4\alpha}+o(1)}, &d=1,\\
t^{-{{1 \over 2\alpha}-{d \over 2}}+o(1)}, &d\ge2
\end{cases}
\label{eq:on-diagonal}
\end{equation}
as $t\to \infty$. If $\E[z(0)]<\infty$, then $\P$-almost surely, 
\begin{align}
 P^\omega_0(X_t=0)= \left({(4\pi)^{-\frac{d+1}{2}}\E[z(0)]^{-\frac12}}+o(1)\right)t^{-\frac{1+d}{2}}
\label{eq:standard-ondiag}
\end{align}
as $t\to\infty$. In particular, when $d=1$, $(X_t)_{t\ge 0}$ is transient if $\alpha<1$ and recurrent if $\E[z(0)]<\infty$, and when $d\ge 2$, it is always transient.
\end{theorem}
\begin{remark}
\begin{enumerate}[leftmargin=20pt]
 \item By irreducibility of the random walk, the return probability $P^\omega_x(X_t=x)$ has the same asymptotics as in~\eqref{eq:on-diagonal} for all $x\in\Z^d$, $\P$-almost surely. This implies that when $\alpha<1$, the weighted graph $(\Z^{1+d}, \omega)$ has the spectral dimension 
\begin{equation}
d_{\text s}= -2\lim_{t\to\infty}\frac{P^\omega_x(X_t=x)}{\log t}=
\begin{cases}
 {3\alpha+1 \over 2\alpha}, &\textrm{if }d=1,\\[5pt]
 {1 \over \alpha}+d, &\textrm{if }d\ge 2,
\end{cases}
\label{eq:spec-dim}
\end{equation}
which is strictly greater than $1+d$. 
 \item The on-diagonal behavior of the heat kernel is often related to the volume (cardinality) of an intrinsic distance ball with radius $t^{\frac12}$. In our setting, the above asymptotics are the same as the volume of the ball with respect to $\bar{d}$ in~\eqref{eq:intrinsic} with radius $t^{\frac12}$. 
\end{enumerate}
\end{remark}
The right-hand side of~\eqref{eq:standard-ondiag} coincides with the asymptotics of $P^{\E[\omega]}_0(X_t=0)$, the return probability for the random walk in the averaged conductance. It is natural to ask when this extends to the local central limit theorem. The following result answers affirmatively when $\alpha>\frac{d}{2}\vee 1$, and negatively when $\alpha<\frac{d}{2}\vee 1$. 
\begin{proposition}
\label{rem:LCLT}
For any $R>0$, the following hold $\P$-almost surely: 
\begin{enumerate}
 \item When $\alpha>\frac{d}{2}\vee 1$, 
\begin{equation}
\lim_{t\to\infty}\sup_{|x|\le R t^{1/2}}\left|\frac{P^{\omega}_0(X_t=x)}{P^{\E[\omega]}_0(X_t=x)}-1\right|=0. 
\end{equation}
 \item When $d\ge 3$ and $\alpha<\frac{d}{2}\vee 1$, 
\begin{equation}
\lim_{t\to\infty}\inf_{|x|\le R t^{1/2}}\frac{P^{\omega}_0(X_t=x)}{P^{\E[\omega]}_0(X_t=x)}=0.
\end{equation}
\end{enumerate}
\end{proposition}

Our next result is about the decay rate of $P^\omega_0(X_t=t^\delta\mathbf{e}_1)$. When $\delta>\frac{s(d,\alpha)}{2}\vee\frac12$, this is a large or moderate deviation probability, and in~\cite[Theorem~4]{DF18}, it is proved to decay stretched exponentially except when $d\ge 3$, $\alpha<\frac{d}{2}$ and $\delta<\frac{d}{4\alpha}$. In the last regime, unlike in the independent and identically distributed conductance case (cf.~\cite{BD10}), it is shown to exhibit a power law decay in~\cite[Theorem~4]{DF18}. The following theorem determines the exponent.
\begin{theorem}
\label{thm:RCM-power}
Let $d\ge 3$, $\alpha<\frac{d}{2}$ and $\delta\in[0,\frac{d}{4\alpha})$. Then, 
\begin{equation}
 P^\omega_0(X_t=t^\delta\mathbf{e}_1)=t^{-r(d,\alpha,\delta)+o(1)} 
\end{equation}
as $t\to\infty$, where 
\begin{equation}
\label{eq:RCM_exponent}
r(d,\alpha,\delta)=
\begin{cases}
\frac{1}{2\alpha}\vee \frac12+\frac{d}{2},& \delta\in \left[0,\frac{1}{2\alpha}\vee \frac12\right],\\
\delta(1+2\alpha)-1+\frac{d}{2},& \delta\in \left(\frac{1}{2\alpha}\vee \frac12,\frac{d}{4\alpha}\right).
\end{cases} 
\end{equation}
\end{theorem}
While this power law decay is new for nearest neighbor walks, it is standard for random walks with unbounded jumps. In our setting, the unbounded conductance along lines play a similar role to a long distance jump. More precisely, the above asymptotics is the same as the probability of the strategy explained in Figure~\ref{fig:highway}.
\begin{figure}[h]
\includegraphics{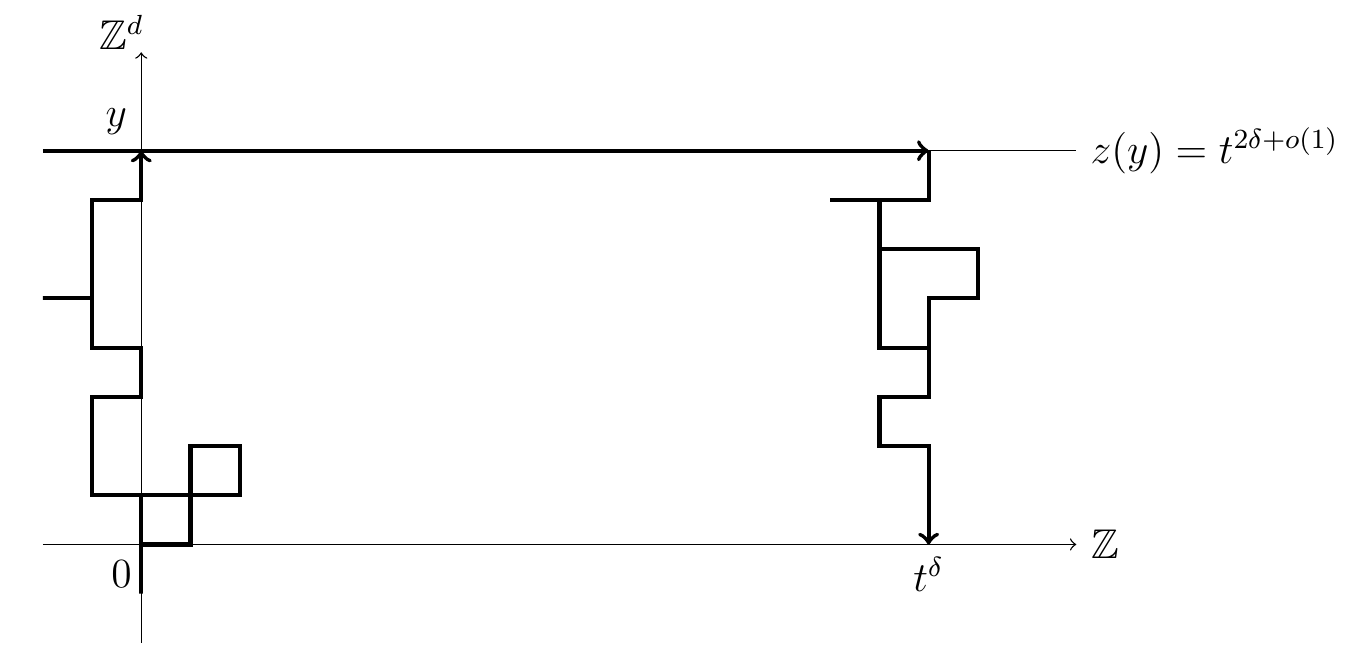}
\caption{A strategy of the random walk that gives the lower bound for $P^\omega_0(X_t=t^\delta \mathbf{e}_1)$. The thick polyline represents the range of the random walk. The second coordinate walk $X^2$ visits $y\in\Z^d$ with $z(y)\ge t^{2\delta}$, spends a unit time there, and then come back to the origin. While the second coordinate is $y$, the first coordinate walk $X^1$ moves in a speed faster than $t^{2\delta}$ and hence can reach $t^{\delta}=\sqrt{t^{2\delta}}$ in the unit time.}
\label{fig:highway}
\end{figure}

Based on the above results on heat kernel asymptotics, we can derive estimates on the Green function defined by 
\begin{equation}
 g^\omega(x,y)=\int_0^\infty P^\omega_x(X_t=y)\dd t.
\label{eq:def_Green}
\end{equation} 
Note that our random walk is transient and $g^\omega$ is well-defined if $\alpha<1$ or $d\ge 2$. When $d\ge 3$, we have the standard exponent for $\alpha\ge \frac{d}{4}\vee 1$, and non-standard exponents for $\alpha< \frac{d}{4}\vee 1$ depending on $d$ and $\alpha$. When $d=2$, the decay exponent is always $-1$, which is the same as for the three dimensional simple random walk, regardless the value of $\alpha$. The non-standard behaviors are caused by the power law decay in Theorem~\ref{thm:RCM-power}.
\begin{theorem}
 \label{thm:Green}
For $\alpha\le \frac{d}{2}\vee 1$, $\P$-almost surely, 
\begin{equation}
 g^\omega(0,n\mathbf{e}_1) = 
\begin{cases}
n^{-\frac{1-\alpha}{1+\alpha}+o(1)},& d=1\text{ and } \alpha<1,\\
n^{-1-(1\wedge \alpha\wedge \frac{4\alpha}{d})(d-2)+o(1)},& d\ge 2
\end{cases}
\end{equation}
as $n\to\infty$. For $d\ge 2$ and $\alpha> \frac{d}{2}$, $\P$-almost surely, 
\begin{equation}
 g^\omega(0,n\mathbf{e}_1) =
 \left({\frac{1}{4}\Gamma(\tfrac{d-1}{2})\E[z(0)]^{\frac{d-2}{2}}}+o(1)\right)n^{1-d}
\label{eq:standard-green}
\end{equation}
as $n\to\infty$. 
\end{theorem}
\begin{remark}
\begin{enumerate}[leftmargin=20pt]
 \item Similarly to Proposition~\ref{rem:LCLT}, the right-hand side of~\eqref{eq:standard-green} coincides with the asymptotics of the Green function of the random walk under $P^{\E[\omega]}$. Moreover, except for the case $d \ge 5$ and $\alpha <\frac{d}{4}$, we have $g^\omega(0,n\mathbf{e}_1)=\bar{d}(0,n\mathbf{e}_1)^{2-d_{\text{s}}+o(1)}$ as $n\to\infty$ by using the distance in~\eqref{eq:intrinsic} and the spectral dimension in~\eqref{eq:spec-dim}. 
 \item In the case $d \ge 5$ and $\alpha < \frac{d}{4}$, the Green function in the direction $\mathbf{e}_1$ decays slower than $\bar{d}(0,n\mathbf{e}_1)^{2-d_{\text{s}}+o(1)}$ as $n\to\infty$. In particular for $\alpha\in(1,\frac{d}{4})$, the exponent is strictly smaller than $1-d$ although the quenched functional central limit theorem holds. This looks strange but there is no contradiction since this slow decay is specific to the $\mathbf{e}_1$-direction. We elaborate more on this singularity in direction in Remark~\ref{rem:Green-off}. 
\end{enumerate}
\end{remark}

Let us compare Theorems~\ref{thm:on-diag} and~\ref{thm:Green} with recent results for more general ergodic random conductance models. The quenched local central limit theorem is proved in~\cite[Theorem~1.11, Remarks~1.5 and~1.12]{ADS16b} under a moment condition which in our case reads as
\begin{equation}
 \E[z(0)^p]<\infty\text{ for some }p>\tfrac{1+d}{2}.
\label{eq:ADS}
\end{equation}
The condition in~\cite{ADS16b} involves a negative moment of random conductance which controls a certain trapping effect. Here we do not need it since the conductance is one except for the first coordinate direction and hence no trapping can occur. 
Proposition~\ref{rem:LCLT} shows that~\eqref{eq:ADS} can be relaxed to $p>\frac{d}{2}$ in our special model.
More recently in~\cite{BS19}, the invariance principle and the elliptic Harnack inequality have been proved under a slightly weaker condition where $\tfrac{1+d}{2}$ is replaced by $\tfrac{d}{2}$. Combining them, one can deduce the convergence of the Green function as in the last part of Theorem~\ref{thm:Green}. Here our result requires exactly the same moment condition. In addition, we show in Remark~\ref{rem:Green-off} that if $\alpha < \frac{d}{2}$ in our model, which implies $\E[z(0)^p]=\infty$ for some $p<\frac{d}{2}$, then the elliptic Harnack inequality fails to hold.  

Let us briefly explain how we use the representation~\eqref{representationRWRS} in the proof. For example, the on-diagonal heat kernel in Theorem~\ref{thm:on-diag} can be related to a negative moment of $A^2(t)$ as follows:
\begin{equation}
\begin{split}
P^\omega_0(X_t=0)
&=P_0\left(S^1_{A^2(t)}= S^2_t=0\right)\\
&=E_0\left[p_{A^2(t)}(0,0); S^2_t=0\right]\\
&= ({(4\pi)^{-1/2}}+o(1))E_0\left[A(t)^{-{1\over 2}}; S_t=0\right]
\end{split}
\label{HK-RWRS}
\end{equation}
as $t\to\infty$. We dropped the superscript in the last formula since it involves only the second random walk $S^2$. If we formally substitute~\eqref{eq:scaling}, that is, 
\begin{equation}
A(t)=t^{s(d,\alpha)+o(1)} \textrm{ with }s(d,\alpha)=
\begin{cases}
{\alpha+1 \over 2\alpha}\vee 1 &\textrm{if }d=1,\\
{1 \over \alpha}\vee 1& \textrm{if }d\ge 2
\end{cases}
\end{equation}
into~\eqref{HK-RWRS}, we obtain Theorem~\ref{thm:on-diag}. Thus our task is to justify this substitution by controlling the upper and lower deviations of $A(t)$ away from the above scaling for $\P$-almost every $z$. In view of~\eqref{HK-RWRS}, the lower and upper bound for $P_0^\omega(X_t=0)$ are related to the upper and lower tail estimate for the random walk in random scenery, respectively. 

Before closing this introduction, we mention two variants of our model. The first is a \emph{multidimensional layer} model which can be covered by our method. 
\begin{remark}
A similar model on $\Z^{d_1+d_2}$ can be defined by setting the conductance around $x=(x_1,x_2)\in \Z^{d_1+d_2}$ to be 
\begin{equation}
  \omega(x,x\pm \mathbf{e}_i)= 
 \begin{cases}
  z(x_2),& i\le d_1,\\
  1,& i> d_1.
 \end{cases}
\label{eq:multidim}
\end{equation}
All the above results have extensions to this setting and we list them below, though we give proofs only in the simplest $\Z^{1+d}$ case for brevity. In the following list, we assume $d_1\ge 2$ and state the results with $o(1)$ in the exponents which is not always necessary. 

\emph{On-diagonal estimate}: For any $d_1\ge 2$, $\P$-almost surely, 
\begin{equation}
P^\omega_0(X_t=0)=t^{-\frac{d_1}{2}s(d_2,\alpha)-\frac{d_2}{2}+o(1)} 
\end{equation}
as $t\to\infty$, which implies that
\begin{equation}
d_{\text s}=d_1s(d_2,\alpha)+d_2\;
\begin{cases}
>d_1+d_2, &\alpha<1\\
=d_1+d_2, &\alpha\ge 1.
\end{cases} 
\end{equation} 
In particular, the random walk is always transient in this case.

\emph{Off-diagonal estimate}:
For any $d_2\ge 3$, $\alpha<\frac{d_2}{2}$, $\delta\in[0,\frac{d_2}{4\alpha})$ and $x_1\in\R^{d_1}\setminus\{0\}$, $\P$-almost surely, 
\begin{equation}
 P^\omega_0(X_t=t^\delta(x_1,0))=t^{-r(d_1,d_2,\alpha,\delta)+o(1)}
\end{equation}
as $t\to\infty$, where 
\begin{equation}
r(d_1,d_2,\alpha,\delta)=
\begin{cases}
d_1\left(\frac{1}{2\alpha}\vee \frac{1}2\right)+\frac{d_2}{2},& \delta\in \left[0,\frac{1}{2\alpha}\vee \frac12\right],\\[5pt]
\delta(d_1+2\alpha)-1+\frac{d_2}{2},& \delta\in \left(\frac{1}{2\alpha}\vee \frac12,\frac{d_2}{4\alpha}\right).
\end{cases} 
\end{equation}

\emph{Green function estimate}: For any $x_1\in\R^{d_1}\setminus\{0\}$, $\P$-almost surely, 
\begin{equation}
 g^\omega(0,n(x_1,0)) = 
\begin{cases}
n^{-d_1+(1\wedge\frac{2\alpha}{\alpha+1})}, &d_2=1,\\
n^{-d_1-(1\wedge \alpha\wedge \frac{4\alpha}{d_2})(d_2-2)+o(1)}, &d_2\ge 2
\end{cases}
\label{eq:Green-MD}
\end{equation}
as $n\to\infty$. 
\end{remark}
The second variant is the constant speed random walk. 
\begin{remark}
For the $\omega$ given in~\eqref{eq:multidim}, consider a continuous time Markov chain $(Y_t)_{t\ge 0}$ on $\Z^{d_1+d_2}$ with generator
\begin{equation}
\widetilde{\mathcal{L}}^\omega f(x)=\frac{1}{\omega(x)}\sum_{|e|=1}\omega(x,x+e)
(f(x+e)-f(x)), 
\end{equation} 
where $\omega(x)=\sum_{|e|=1}\omega(x,x+e)=2d_1z(x)+d_2$. This process jumps to a neighboring site with probability proportional to $\omega(x,\cdot)$ and with rate one, hence called the constant speed random walk. This can be realized as a time change of the variable speed random walk $(X_t)_{t\ge 0}$ as follows: if we define $B(t)=\int_0^t \omega(X_u)\dd u=2d_1A^2(t)+2d_2t$, then by using its right continuous inverse $B^{-1}$, we have
\begin{equation}
 Y_t \stackrel{\dd}{=} X_{B^{-1}(t)} =(S^1_{A^2(B^{-1}(t))}, S^2_{B^{-1}(t)}). 
\label{eq:CSRW}
\end{equation}
In the last expression, $A^2(B^{-1}(t))$ is comparable to $t$, up to multiplicative constant, and hence the first coordinate behaves more or less like a simple random walk on $\Z^{d_1}$ for large time. The second coordinate is a well-studied process called the Bouchaud trap model. For quenched results, we refer the reader to~\cite{BAC07} for functional limit theorems and~\cite{CHK19} for a two-sided heat kernel estimate. Though the long time behaviors in Theorems~\ref{thm:on-diag} and~\ref{thm:RCM-power} become different by the time change, the decay of Greens function in Theorem~\ref{thm:Green} is unchanged since $\tilde{g}^\omega(x,y)=g^\omega(x,y)\omega(y)$. This allows us to formally compute the spectral dimension for the constant speed random walk on $\Z^{d_1+d_2}$ as
\begin{equation}
\begin{cases}
d_1+\left(1\wedge \frac{2}{\alpha+1}\right), & d_2=1, \text{ and $\alpha<1$ if $d_1=1$},\\[5pt]
d_1+2+\bigl(1\wedge \alpha \wedge \frac{4\alpha}{d_2}\bigr)(d_2-2), & d_2\ge 2.
\end{cases}
\label{eq:d_s-CSRW}
\end{equation}
The detail is given in Remark~\ref{rem:Green-off} (iii). 
\end{remark}
\subsection{Notation convention}
In the proofs, we will use $c$ and $c'$ to denote positive constants depending only on $d$ and $\alpha$, whose values may change from line to line. We write $F(t)\sim G(t)$ as $t\to\infty$ to indicate $\lim_{t\to\infty}F(t)/G(t)= 1$, and $F(t)\asymp G(t)$ as $t\to\infty$ to indicate $c\le \liminf_{t\to\infty}F(t)/G(t)<\limsup_{t\to\infty}F(t)/G(t)<c'$. 
\section{Overview of the paper}
Since we have various results depending on the parameters, the proofs split into many cases. We summarize the organization of the rest of the paper and also explain basic ideas of the proofs.

In Section~\ref{sec:HK}, we recall standard estimates on the transition probability for the simple random walk. 

In Section~\ref{sec:power-law}, we prove Theorem~\ref{power-law}, that is, the power law decay of the random walk in random scenery. We first prove that for the random walk to achieve $A(t)\ge t^\rho$, it is almost necessary and sufficient to visit the relevant level set $\{x\in\Z^d\colon z(x)\ge t^\rho\}$. More precisely, we prove that it is too difficult to get a contribution from lower level sets of $z$. Here we need to invoke some estimates from~\cite{DF18}. The upper bound on the hitting probability to the relevant level set is shown by a rather simple argument using the asymptotics of the random walk range. The lower bound is obtained by using the so-called second moment method. 

In Section~\ref{sec:lower}, we prove the on-diagonal lower bounds in Theorem~\ref{thm:on-diag}. Essentially, we simply substitute the following results on the random walk in random scenery into~\eqref{HK-RWRS}: 
the law of large numbers for $A(t)$ when $\E[z(0)]<\infty$, Theorem~\ref{RWRS} when $d=1,2$ and $\E[z(0)]=\infty$, and Theorem~\ref{power-law} when $d\ge 3$ and $\E[z(0)]=\infty$.

In Section~\ref{sec:upper}, we prove the on-diagonal upper bounds in Theorem~\ref{thm:on-diag}. In fact they follow immediately from Proposition~\ref{prop:RWRSlower} and most of the section is devoted to the proof of it. When $\alpha\ge 1$, we use certain truncations to reduce the problem to upper and lower deviations for the random walk in bounded scenery, which are rather well-understood. For the case $\alpha<1$, it essentially amounts to proving that it is difficult to 
\begin{enumerate}
 \item reduce the local time on the level set $\{z(x)\ge t^{\frac{1}{2\alpha}}\}$ to $o(t^{1/2})$ when $d=1$, 
 \item make the random walk avoid the level set $\{z(x)\ge t^{\frac{1}{\alpha}}\}$ when $d\ge 2$.
\end{enumerate}
The argument for $d=1$ is rather bare-handed and based on the path decomposition according to the successive moves over the points in the above level set. For the case $d\ge 2$, we use the so-called method of enlargement of obstacles developed by Sznitman~\cite{Szn98}. 

Sections~\ref{sec:LCLT},~\ref{sec:RCM-power} and~\ref{sec:Green} are devoted to the proofs of Proposition~\ref{rem:LCLT}, Theorems~\ref{thm:RCM-power} and~\ref{thm:Green}, respectively. The proofs are mostly straightforward applications of the results in the earlier sections.  
\section{A bound on the continuous time random walk}
\label{sec:HK}
We frequently use the following estimate on the transition probability of the
continuous time simple random walk
$p_t(x,y)=P_x(S_t=y)$.
This can be found in~\cite[Propositions~4.2 and~4.3]{DD05}. 
\begin{lemma}
\label{lem:RWHK}
There exist positive constants $c_1$--$c_4$ such that when $t\ge 1$, 
\begin{equation}
c_1t^{-\frac{d}2}\exp\left\{-c_2\frac{|x|^2}t\right\}
\le p_t(0,x)\le
c_3t^{-\frac{d}2}\exp\left\{-c_4\frac{|x|^2}t\right\}
\end{equation}
for $|x|\le t$ and 
\begin{equation}
 \exp\left\{-c_2|x|\left(1\vee\log\frac{|x|}t\right)\right\}
\le p_t(0,x)\le
 \exp\left\{-c_4|x|\left(1\vee\log\frac{|x|}t\right)\right\}
\end{equation}
for $|x|> t$.
\end{lemma}
\section{Power law decay rate for random walk in random scenery}
\label{sec:power-law}
\begin{proof}[Proof of Theorem~\ref{power-law}]
We first prove the upper bounds. Let us define the level set for $\lambda>0$ by
\begin{equation}
 \calH_{\lambda}(t)=\set{x\in \Z^d \cap [-t^{1/2+\epsilon},t^{1/2+\epsilon}]^d \colon z(x)\ge {t^{\lambda}}}.
\end{equation}
We often write $\calH_{\lambda}$ instead of $\calH_{\lambda}(t)$ for simplicity and let $\sfH_{\calH_\lambda}$ denote the hitting time to $\calH_\lambda$. 
Then we can write 
\begin{equation}
 P_0\left(A(t)\ge t^{\rho}\right)
 \le P_0\left(A(t)\ge t^{\rho}, \sfH_{\calH_{{\rho-5\epsilon}}}\ge t\right)
+P_0\left(\sfH_{\calH_{{\rho-5\epsilon}}}<t\right).
\label{eq:high-low}
\end{equation}
We are going to show that the first term in~\eqref{eq:high-low} decays stretched exponentially in Lemma~\ref{lem:low} and that the second term obeys the right power law decay in Lemma~\ref{lem:highest}. 
\begin{lemma}
\label{lem:low}
Let $d\ge 3$, $\alpha<d/2$ and $\rho\in(\tfrac1\alpha\vee 1, \tfrac{d}{2\alpha}]$. For any $\epsilon\in(0,\tfrac{\rho}{5}\wedge (\tfrac{1-\alpha\rho}{3}\vee\tfrac{\rho-1}{3}))$, 
there exists $c(\epsilon)>0$ such that $\P$-almost surely, for all sufficiently large $t$, 
\begin{equation}
P_0\left(A(t)\ge t^\rho, \sfH_{\calH_{{\varrho-5\epsilon}}}\ge t\right)
\le \exp\left\{-t^{c(\epsilon)}\right\}.
\end{equation}
\end{lemma}
\begin{proof}
We use some estimates from~\cite{DF18}. 
Let $\ell_t(\cdot)=\int_0^t 1_{S_u}(\cdot){\dd}u$ be the occupation time measure for the simple random walk. 
Then we have an obvious bound (cf.~\cite[eq.~(49)]{DF18})
\begin{equation}
\begin{split}
 P_0\left(A(t)\ge t^\rho, \sfH_{\calH_{{\rho-5\epsilon}}}\ge t\right)
 &\le \sum_{k=0}^{\lfloor \rho/\epsilon \rfloor-4}P_0\left(\ell_t(\calH_{{k\epsilon}}\setminus \calH_{{(k+1)\epsilon}})\ge \frac{1}{\lfloor \rho/\epsilon \rfloor}t^{\rho-(k+1)\epsilon}\right)\\
&\quad +  P_0\left(\max_{0\le u \le t} |S_u|\ge t^{{1 \over 2}+\epsilon}\right).
\end{split}
\label{eq:level-decomp}
\end{equation}
The second term on the right-hand side is easily seen to decay stretched exponentially in $t$ by the reflection principle and Lemma~\ref{lem:RWHK}. For the first term on the right-hand side, we can drop the summands with $k<(\rho-1)/\epsilon-1$ since the total mass of $\ell_t$ is $t$. In order to control the other summands, we introduce $\eta=1-\rho+(k+3)\epsilon$. Then we can verify that the condition $\eta/\alpha<k\epsilon$ in~\cite[Lemma~2]{DF18} holds for all $k\le \lfloor \rho/\epsilon \rfloor-4$ and $\epsilon\le (1-\alpha\rho)/3$ as follows: 
\begin{itemize}
 \item when $\alpha<1$, this is equivalent to $k\epsilon<(\rho-1-3\epsilon)/(1-\alpha)$, which holds true for the largest $k=\lfloor \rho/\epsilon \rfloor-4$; 
 \item when $\alpha\ge 1$, this is equivalent to $\rho>1+(1-\alpha)k\epsilon+3\epsilon$ and it holds for any $k\in \N$ under our assumption $\epsilon<\frac{\varrho-1}{3}$. 
\end{itemize}
Once the condition of~\cite[Lemma~2]{DF18} is verified, we can follow the same argument as in~\cite[eq.~(74)]{DF18} to show that 
\begin{equation}
 P_0\left(\ell_t(\calH_{{k\epsilon}}\setminus \calH_{{(k+1)\epsilon}})\ge \frac{1}{\lfloor \rho/\epsilon \rfloor}t^{\rho-(k+1)\epsilon}\right)
\le \exp\left\{-ct^{\rho-(k+3)\epsilon}\right\}
\end{equation}
for $\epsilon\le (1-\alpha\rho)/3\vee (\rho-1)/3$. Since this decays stretched exponentially in $t$ for all $k\le \lfloor \rho/\epsilon \rfloor-4$, we are done.
\end{proof}
We can prove the following version by almost the same argument.
\begin{lemma}
\label{lem:low2}
Let $d\ge 3$, $1<\alpha<d/2$ and $c>\E[z(0)]$. For any $\epsilon\in(0,\frac{\alpha-1}{20(\alpha+2)})$, there exists $c(\epsilon)>0$ such that $\P$-almost surely, for all sufficiently large $t$, 
\begin{equation}
P_0\left(A(t)\ge ct, \sfH_{\calH_{{1-5\epsilon}}}\ge t\right)
\le \exp\left\{-t^{c(\epsilon)}\right\}.
\end{equation}
\end{lemma}
\begin{proof}
We mostly repeat the proof of Lemma~\ref{lem:low} with $\rho$ replaced by 1. The only difference is that the summands in~\eqref{eq:level-decomp} are positive for all $k\ge 0$. We instead introduce $c'\in (\E[z(0)],c)$ and write
\begin{equation}
\begin{split}
  P_0\left(A(t)\ge ct\right)
&\le  P_0\left(\int_0^t z(S_u)\wedge t^{l\epsilon}{\dd}u
 \ge c't\right)\\
&\quad+\sum_{k=l}^{\lfloor 1/\epsilon \rfloor-4}
 P_0\left(\ell_t(\mathcal{H}_k\setminus \mathcal{H}_{k+1})
 \ge \frac{c-c'}{\lfloor 1/\epsilon \rfloor}t^{1-(k+1)\epsilon}\right)\\
&\quad +P_0\left(\max_{0\le u\le t}|S_u| > t^{\frac12+\epsilon}\right).
 \label{slicing2}
\end{split}
\end{equation}
The last term decays stretched exponentially just as before. 
Next, if we let $l$ be the smallest integer larger than $\frac{3}{\alpha-1}$, then for any $k\ge l$, one can check that the probability 
\begin{equation}
P_0\left(\ell_t(\calH_{{k\epsilon}}\setminus \calH_{{(k+1)\epsilon}})\ge \frac{c-c'}{\lfloor 1/\epsilon \rfloor}t^{1-(k+1)\epsilon}\right) 
\end{equation}
decays stretched exponentially (see~\cite[eq.~(75)]{DF18} and the following discussion therein). Finally, note that our assumption on $\epsilon$ and the choice of $l$ imply that $l\epsilon<\frac{1}{20}$. Then~\cite[eq.~(76)]{DF18} shows that the first term in~\eqref{slicing2} decays stretched exponentially. 
\end{proof}
\begin{lemma}
\label{lem:highest}
Let $d\ge 3$ and $\rho\in (\tfrac{1}{\alpha}\vee 1, \tfrac{d}{2\alpha}]$. Then $\P$-almost surely, 
\begin{equation}
P_0\left(\sfH_{\calH_{{\rho-5\epsilon}}}<t\right) 
\le t^{-\alpha\rho+1+6\alpha\epsilon+o(1)} \textrm{ as }t\to\infty. 
\end{equation}
for all sufficiently small $\epsilon>0$. When $\alpha>1$, this bound remains valid for $\rho=1$. 
\end{lemma}
\begin{proof}
We start by introducing an unrestricted version of the level set
\begin{equation}
 \overline\calH_{{\rho-5\epsilon}}(t)=\set{x\in \Z^d \colon z(x)\ge {t^{{\rho-5\epsilon}}}} \supset \calH_{{\rho-5\epsilon}}(t).
\end{equation}
Then we have 
\begin{equation}
\begin{split}
\P\otimes P_0\left(\sfH_{\overline\calH_{{\rho-5\epsilon}}(t)}\ge {2t}\right)
& =\P\otimes P_0\left(\textrm{every }x\in S_{[0,{2t}]} \textrm{ does not belong to }\calH_{{\rho-5\epsilon}}\right)\\
&= E_0\left[\P\left(x \textrm{ does not belong to }\overline\calH_{{\rho-5\epsilon}}(t)\right)^{|S_{[0,{2t}]}|}\right]\\
&= E_0\left[\left(1-{t^{-\alpha\rho+5\alpha\epsilon+o(1)}}\right)^{|S_{[0,{2t}]}|}\right]
\end{split}
\end{equation}
as $t\to\infty$. By using the Jensen inequality and that $E_0[|S_{[0,2t]}|]$ grows linearly as $t\to\infty$ (see~\cite[Theorem~1]{DE51}),
we can further bound the above probability from below by 
\begin{equation}
\begin{split}
\P\otimes P_0\left(\sfH_{\overline\calH_{{\rho-5\epsilon}}(t)}\ge {2t}\right)
 &\ge \left(1-{t}^{-\alpha\rho+5\alpha\epsilon+o(1))}\right)^{E_0[|S_{[0,{2t}]}|]}\\
&\ge \exp\left\{-{t}^{-\alpha\rho+1+5\alpha\epsilon+o(1)}\right\}\\
&\ge 1-{t}^{-\alpha\rho+1+5\alpha\epsilon+o(1)}
\end{split}
\end{equation}
as $t\to \infty$. Then by the Markov inequality, we obtain
\begin{equation} 
\begin{split}
\P\left(P_0\left(\sfH_{\overline\calH_{{\rho-5\epsilon}}(t)}< {2t}\right)
> {t}^{-\alpha\rho+1+6\alpha\epsilon}\right)
& \le t^{\alpha\rho-1-6\alpha\epsilon}\P\otimes P_0\left(\sfH_{\overline\calH_{{\rho-5\epsilon}}(t)}< {2t}\right)\\
& \le {t}^{-\alpha\epsilon+o(1)}
\end{split}
\end{equation}
as $t\to\infty$. Substituting $t=2^n$ ($n\in\N$) and using the Borel--Cantelli lemma, we conclude that $\P$-almost surely,
\begin{equation}
 P_0\left(\sfH_{\overline\calH_{{\rho-5\epsilon}}{(2^n)}}<2^{n+1}\right)
\le 2^{-n(\alpha\rho-1-6\alpha\epsilon+o(1))} 
\end{equation}
for all sufficiently large $n$. Since $P_0\left(\sfH_{\calH_{{\rho-5\epsilon}}{(t)}}<t\right)$ for $t\in [2^{n}, 2^{n+1})$ is bounded by the above left-hand side, we are done.
\end{proof}
The above three lemmas yield the upper bound in~\eqref{eq:unpinned}.
Let us discuss how to include the pinning restriction $S_t=0$. Our starting point is
\begin{equation}
\begin{split}
&P_0\left(A(t)\ge t^\rho, S_t=0\right)\\
&\quad \le P_0\left(A(\tfrac{t}{2})\ge \tfrac{1}{2}t^\rho, S_t=0\right)
 +P_0\left(A(\tfrac{t}{2})\circ\theta_{t/2} \ge \tfrac{1}{2}t^\rho, S_t=0\right)\\
&\quad =2 P_0\left(A(\tfrac{t}{2})\ge \tfrac{1}{2}t^\rho, S_t=0\right),
\end{split}
\label{divide-mid1}
\end{equation}
where $\theta$ denotes the time shift operator and we have used the time reversal. By using the Markov property at time $t/2$ and Lemma~\ref{lem:RWHK}, we obtain
\begin{equation}
\begin{split}
P_0\left(A(\tfrac{t}{2})\ge \tfrac{1}{2}t^\rho, S_t=0\right)
&\le E_0\left[p_{t/2}(S_{t/2},0); A(\tfrac{t}{2})\ge \tfrac{1}{2}t^\rho\right]\\
&\le ct^{-\frac{d}{2}} P_0\left(A(\tfrac{t}{2})\ge \tfrac{1}{2}t^\rho\right). 
\end{split}
\label{divide-mid2}
\end{equation}
Substituting the upper bound in~\eqref{eq:unpinned}, we conclude the upper bound in~\eqref{eq:lbd_pinned}. 
\begin{remark}
\label{rem:coverage}
The above argument shows that the upper bound in~\eqref{eq:lbd_pinned} holds for $\rho\ge \frac{d}{2\alpha}$ as well. For $\rho> \frac{d}{2\alpha}$, Theorem~\ref{RWRS} gives a better bound but we need this boundary coverage in the proof of Theorem~\ref{thm:RCM-power}.
\end{remark}

Next we proceed to prove the lower bounds. We only prove the lower bound in~\eqref{eq:lbd_pinned} since the argument essentially contains the proof of~\eqref{eq:unpinned}. Recall that $\ell_t$ denotes the occupation time measure for $(S_u)_{u\in[0,t]}$ and let
\begin{equation}
 \underline{\calH}_{{\rho}}=\calH_{{\rho}}\cap B
(0,t^{\frac{1}{2}}
) \setminus B
(0,t^{\frac{\alpha\rho}{d}+\epsilon}
),
\end{equation}
where we choose $\epsilon>0$ so small that $\frac{\alpha\rho}{d}+\epsilon<\frac12$. Then it suffices to show that
\begin{equation}
 P_0\left(\ell_t(\underline{\calH}_{{\rho}})\ge 1, S_t=0\right)\ge t^{-\alpha\rho+1-\frac{d}{2}+o(1)}.\label{eq:goal-l_t}
\end{equation} 
To this end, we first bound the probability that the random walk visits $\underline{\calH}_{{\rho}}$ before $t/2$, by using the so-called second moment method. We introduce the annulus
\begin{equation}
 D(k)=B
(0,(k+1)t^{\frac{\alpha\rho}{d}+\epsilon}
)\setminus B
(0,kt^{\frac{\alpha\rho}{d}+\epsilon}
).
\end{equation}
Then a simple argument using the Borel--Cantelli lemma yields that $\P$-almost surely, 
\begin{equation}
 \left|\calH_{{\rho}}\cap D(k)\right|\in \left(k^{d-1}t^{(d-1)\epsilon},k^{d-1}t^{(d+1)\epsilon}\right) \quad \text{for }1\le k \le t^{\frac12-\frac{\alpha\rho}{d}-\epsilon}
\label{eq:num_annuli}
\end{equation}
for all sufficiently large $t>0$. 
\medskip

\noindent\underline{First moment}: By Fubini's theorem, we can write
\begin{equation}
E_0\left[\ell_{t/2}(\underline{\calH}_{{\rho}})\right] 
= \sum_{x\in\underline{\calH}_{{\rho}}}\int_0^{t/2} p_u(0,x)\dd u.
\end{equation}
Let us recall a well-known bound $\int_0^{t/2} p_u(0,x)\dd u\asymp |x|^{2-d}$ as $t\to \infty$ which holds uniformly in $|x|\le t^{1/2}$, and hence in $x\in \underline{\calH}_{{\rho}}$ as well (this can be proved by the same argument as for~\cite[Theorem~4.3.1]{LL10}). Then by dividing $\underline{\calH}_{{\rho}}$ into $\{\underline{\calH}_{{\rho}}\cap D(k)\}_{k\ge 1}$ and using~\eqref{eq:num_annuli}, we get
\begin{equation}
\begin{split}
 E_0\left[\ell_{t/2}(\underline{\calH}_{{\rho}})
\right] 
& \asymp \sum_{k=1}^{t^{1/2-\alpha\rho/d-\epsilon}}\sum_{x\in\underline{\calH}_{{\rho}}\cap D(k)} |x|^{2-d}\\
&\begin{cases}
 \ge t^{-\alpha\rho+1-C\epsilon},\\
 \le t^{-\alpha\rho+1+C\epsilon}.
\end{cases}
\end{split}
\label{eq:1st-moment}
\end{equation}
\medskip

\noindent\underline{Second moment}: To bound the second moment, we first write
\begin{equation}
\begin{split}
 E_0\left[\ell_{t/2}(\underline{\calH}_{{\rho}})^2\right]
&= 2\int_0^{t/2}\int_u^{t/2} \sum_{x,y\in\underline{\calH}_{{\rho}}} p_u(0,x)p_{v-u}(x,y) \dd v\dd u\\
& \le \sum_{x\in\underline{\calH}_{{\rho}}} 
\int_0^t p_u(0,x) \int_0^t 
\sum_{y\in\calH_{{\rho}}}p_v(x,y)\dd v \dd u.
\end{split}
\label{eq:2nd-moment}
\end{equation}
By~\cite[Lemma~2]{DF18}, we know that
\begin{equation}
\sup_{|x|\le t^{1/2}}\sum_{y\in\calH_{{\rho}}}\int_0^t p_v(x,y)\dd v
\le t^\epsilon.
\end{equation}
Substituting this bound into~\eqref{eq:2nd-moment} and using~\eqref{eq:1st-moment}, we find the bound
\begin{equation}
 E_0\left[\ell_{t/2}(\underline{\calH}_{{\rho}})^2\right] 
\le t^{-\alpha\rho+1+C\epsilon}.  
\label{eq:2nd_result}
\end{equation}

From~\eqref{eq:1st-moment},~\eqref{eq:2nd_result} and the Paley--Zygmund inequality, it follows that 
\begin{equation}
\begin{split}
P_0\left(\sfH_{\underline{\calH}_{{\rho}}}<t/2\right)
&=P_0\left(\ell_{t/2}(\underline{\calH}_{{\rho}})>0\right)\\
&\ge t^{-\alpha\rho+1-C\epsilon}. 
\end{split}
\end{equation}
Finally, using the strong Markov property and Lemma~\ref{lem:RWHK} together with $|x|\le t^{1/2}$ for $x\in \underline{\calH}_{{\rho}}$, we obtain
\begin{equation}
\begin{split}
& P_0\left(\ell_t(\underline{\calH}_{{\rho}})\ge 1, S_t=0\right)\\
&\quad \ge P_0\left(\sfH_{\underline{\calH}_{{\rho}}}<t/2, \text{there is no jump on }[\sfH_{\underline{\calH}_{{\rho}}},\sfH_{\underline{\calH}_{{\rho}}}+1], S_t=0\right)\\
&\quad \ge t^{-\alpha\rho+1-\frac{d}{2}-C\epsilon}. 
\end{split}
\end{equation}
Since $\epsilon>0$ is arbitrary, the desired bound~\eqref{eq:goal-l_t} follows. 
\end{proof}
\section{On-diagonal lower bounds}
\label{sec:lower}
In this section, we prove the lower bounds in Theorem~\ref{thm:on-diag}. Recall the representation~\eqref{HK-RWRS}.
\subsection{Lower bound under $\E[z(0)]<\infty$}
We first deal with the simplest case $\E[z(0)]<\infty$. In this case, let us fix $M>\E[z(0)]$ and bound the return probability as follows:
\begin{equation}
\begin{split}
P^\omega_0(X_t=0)
&\ge {\left((4\pi)^{-1/2}+o(1)\right)} E_0\left[A(t)^{-{1\over 2}} ; S_t=0,
A(t)< Mt\right]\\
&\ge {\left((4\pi)^{-1/2}+o(1)\right)}(Mt)^{-{1\over 2}}\left(P_0(S_t=0)-P_0\left(A(t)\ge Mt, S_t=0\right)\right)
\end{split}
\label{eq:lbd_finer}
\end{equation}
as $t\to\infty$. 
The last probability is bounded by
\begin{equation}
P_0\left(A(t)\ge Mt, S_t=0\right)
\le ct^{-\frac{d}{2}} P_0\left(A(\tfrac{t}{2})\ge \tfrac{M}{2}t\right)
\label{eq:LLN_pinned}
\end{equation}
just as in~\eqref{divide-mid1}--\eqref{divide-mid2}. This right-hand side is $o(t^{-d/2})$ thanks to the law of large numbers for the random walk in random scenery. Coming back to~\eqref{eq:lbd_finer} and recalling that $M>\E[z(0)]$ is arbitrary, we conclude that 
\begin{equation}
 P^\omega_0(X_t=0)\ge {\left((4\pi)^{-\frac{d+1}{2}}\E[z(0)]^{-1/2}+o(1)\right)}t^{-\frac{1+d}{2}}.
\end{equation}
\subsection{Lower bound for $d=1,2$ with $\E[z(0)]=\infty$}
In order to prove the lower bounds for $d=1,2$ with $\E[z(0)]=\infty$ (hence $\alpha\le 1$), we recall~\eqref{HK-RWRS} to bound the return probability from below by
\begin{equation}
\begin{split}
P^\omega_0(X_t=0)
&\ge c E_0\left[A(t)^{-{1 \over 2}}\wedge 1 ; S_t=0,
A(t)< t^{s(d,\alpha)+\epsilon}\right]\\
&\ge ct^{-{s(d,\alpha)\over 2} -{\epsilon\over 2}}\left(P_0(S_t=0)-P_0\left(A(t)\ge t^{s(d,\alpha)+\epsilon}\right)\right)\\
&\ge ct^{-{s(d,\alpha)\over 2}-{\epsilon\over 2}}\left(c t^{-{d\over 2}}-P_0\left(A(t)\ge t^{s(d,\alpha)+\epsilon}\right)\right).
\end{split}
\label{eq:lbd12}
\end{equation}
Since Theorem~\ref{RWRS} implies
\begin{equation}
\label{upper-tail}
 P_0\left(A(t)\ge t^{s(d,\alpha)+\epsilon}\right)=o\left(t^{-{d\over 2}}\right)
\end{equation}
in this case, the desired lower bound follows. 

\subsection{Lower bound  for $d\ge 3$ with $\E[z(0)]=\infty$}
\label{sec:lbd3}
It remains to deal with the case $d\ge 3$ and $\E[z(0)]=\infty$. As before, we write
\begin{equation}
\begin{split}
P^\omega_0(X_t=0)
&\ge c E_0\left[A(t)^{-{1\over 2}}\wedge 1 ; S_t=0,
A(t)< t^{{1\over \alpha}+2\epsilon}\right]\\
&\ge ct^{-{1\over 2\alpha} -\epsilon}\left(P_0(S_t=0)-P_0\left(A(t)\ge t^{{1\over \alpha}+2\epsilon}, S_t=0\right)\right).
\end{split}
\end{equation}
Since the probability in the last line is $o(t^{-d/2})$ by Theorem~\ref{power-law}, we are done.
\section{On-diagonal upper bounds}
\label{sec:upper}
In this section, we prove the upper bounds in Theorem~\ref{thm:on-diag}. Assuming Proposition~\ref{prop:RWRSlower}, we can deduce the desired upper bound as follows: 
\begin{equation}
\begin{split}
P^\omega_0(X_t=0)
&\sim {(4\pi)^{-1/2}}E_0\left[A(t)^{-{1\over 2}} ; S_t=0,A(t)> t^{s(d,\alpha)-2\epsilon}\right]+\exp\left\{-t^{c(\epsilon)}\right\}\\
&\le {(4\pi)^{-1/2}}t^{-{s(d,\alpha)\over 2} + \epsilon}P_0(S_t=0)+\exp\left\{-t^{c(\epsilon)}\right\}\\
&\le {(4\pi)^{-\frac{d+1}{2}}}t^{-{s(d,\alpha)\over 2}-{d\over 2} + \epsilon}+\exp\left\{-t^{c(\epsilon)}\right\}
\end{split}
\end{equation}
as $t\to\infty$. 
If $\E[z(0)]<\infty$, then we can replace $t^{s(d,\alpha)-2\epsilon}$ in the first line by $t(\E[z(0)]-\epsilon)$ and the finer asymptotics follows. 

Therefore it remains to prove Proposition~\ref{prop:RWRSlower}. The proof splits into four cases: (i) $\E[z(0)]<\infty$, (ii) $\alpha= 1$ and $\E[z(0)]=\infty$, (iii) $d=1$ and $\alpha<1$, and (iv) $d\ge 2$ and $\alpha<1$. 

\subsection{Upper bound under $\E[z(0)]<\infty$}
For any $\epsilon>0$, let us take $M>0$ so large that $\E[z(0)\wedge M]\ge \E[z(0)]-\epsilon$ and define 
\begin{equation}
 V_M(t)=\int_0^t M-z(S_u)\wedge M \dd u. 
\end{equation}
It is then easy to check that $A(t) <t(\E[z(0)]-2\epsilon)$ implies 
\begin{equation}
\begin{split}
  V_M(t)&> t(M-\E[z(0)]+2\epsilon)\\
&\ge t(\E[M-z(0)\wedge M]+\epsilon).
\end{split}
\end{equation}
This is an upper deviation for the random walk in random scenery and (76) in~\cite{DF18} shows that for any $\epsilon>0$, $\P$-almost surely, 
\begin{equation}
P_0\left(V_M(t)\ge t(\E[M-z(0)\wedge M]+\epsilon)\right)\le \exp\left\{-t^{1-\epsilon}\right\}
\end{equation}
for all sufficiently large $t$. Therefore we obtain the desired bound
\begin{equation}
P_0\left(A(t) <t(\E[z(0)]-2\epsilon)\right)\le \exp\left\{-t^{1-\epsilon}\right\}.
\label{eq:boundedRWRS}
\end{equation}
\begin{remark}
The process $V_M$ is a random walk in bounded scenery. Asselah--Castell~\cite{AC03b} proved a large deviation principle with rate $t(\log t)^{-2/d}$ for the Brownian motion in bounded scenery. Although not stated in~\cite{AC03b}, the rate function is positive except at $\E[M-z(0)\wedge M]$ (A.~Asselah and F.~Castell, personal communication, March 2, 2019). 
This improves the bound in~\eqref{eq:boundedRWRS} to $\exp\{-ct(\log t)^{-2/d}\}$. We include the proof of above weaker result for the sake of completeness. 
\end{remark}
\subsection{Upper bound for $\alpha= 1$ with $\E[z(0)]=\infty$}
Recall that we have $s(d,\alpha)=1$. Let us define Bernoulli random variables by
\begin{equation}
 \tilde{z}(x)= 1_{z(x)\ge 1} \le z(x)
\end{equation}
and let $\tilde{A}(t)=\int_0^t\tilde{z}(S_u)\dd u$. Then we know 
\begin{equation}
\label{eq:Laplace}
\begin{split}
  E_0\left[\exp\left\{-A(t)\right\}\right]
 &\le E_0\left[\exp\left\{-\tilde{A}(t)\right\}\right]\\
 &\le \exp\left\{-c\frac{t}{(\log t)^{2/d}}\right\}
\end{split}
\end{equation}
for some $c>0$. This bound is first proved in the continuous setting in~\cite{Szn93c}. See~\cite{BK01b} for a result in more general setting and~\cite{Fuk09b} for a simple argument to derive~\eqref{eq:Laplace} from the result in~\cite{DV79}. By~\eqref{eq:Laplace} and Chebyshev's inequality, we find
\begin{equation}
\begin{split}
P_0(A(t)\le t^{1-\epsilon})
&\le \exp\left\{t^{1-\epsilon}\right\}E_0\left[\exp\left\{-A(t)\right\}\right]\\
&\le \exp\left\{-\frac{c}{2}\frac{t}{(\log t)^{2/d}}\right\}
\end{split}
\end{equation}
for sufficiently large $t$, and Proposition~\ref{prop:RWRSlower} follows in this case. 
\subsection{Upper bound for $\alpha<1$ and $d=1$}
In this case, we have $s(d,\alpha)=\frac{\alpha+1}{2\alpha}$. Let us recall the notation
\begin{equation}
\calH_{{1/2\alpha-\epsilon}}=\set{x\in \Z\cap [-t^{1/2+\epsilon},t^{1/2+\epsilon}]
 \colon z(x)\ge t^{{1\over 2\alpha}-\epsilon}},
\end{equation}
from which the main contribution to $A(t)$ is coming. 
In order to prove Proposition~\ref{prop:RWRSlower} in this case, it suffices to show the following proposition:
\begin{proposition}
\label{local-time}
For sufficiently small $\epsilon>0$ depending on $\alpha$, there exists 
$c>0$ such that $\P$-almost surely, 
\begin{equation}
 P_0\left(\ell_t(\calH_{{1/2\alpha-\epsilon}})\le t^{{1\over 2}-\epsilon}\right)
\le \exp\{-ct^{\epsilon}\}
 \label{lower-deviation1}
\end{equation}
for all sufficiently large $t$. 
\end{proposition}
In order to estimate the probability~\eqref{lower-deviation1}, we introduce
the successive times of returns to/departures from $\calH_{{1/2\alpha-\epsilon}}$. Set $\sfR_0=\sfD_0=0$ and for $k\ge 1$,  
\begin{align}
\sfR_k&=\inf\set{u\ge \sfD_{k-1}\colon S_u\in\calH_{{1/2\alpha-\epsilon}}},\\
\sfD_k&=\inf\set{u\ge \sfR_k\colon S_u\not\in\calH_{{1/2\alpha-\epsilon}}}. 
\end{align}
See also Figure~\ref{fig:RandD}. We are going to show that the number of returns before time $t$, 
which we denote by 
\begin{equation}
 N_t=\sup\set{k\ge 1\colon \sfR_k <t},
\end{equation}
cannot be too small. 
\begin{figure}
 \centering
 \includegraphics{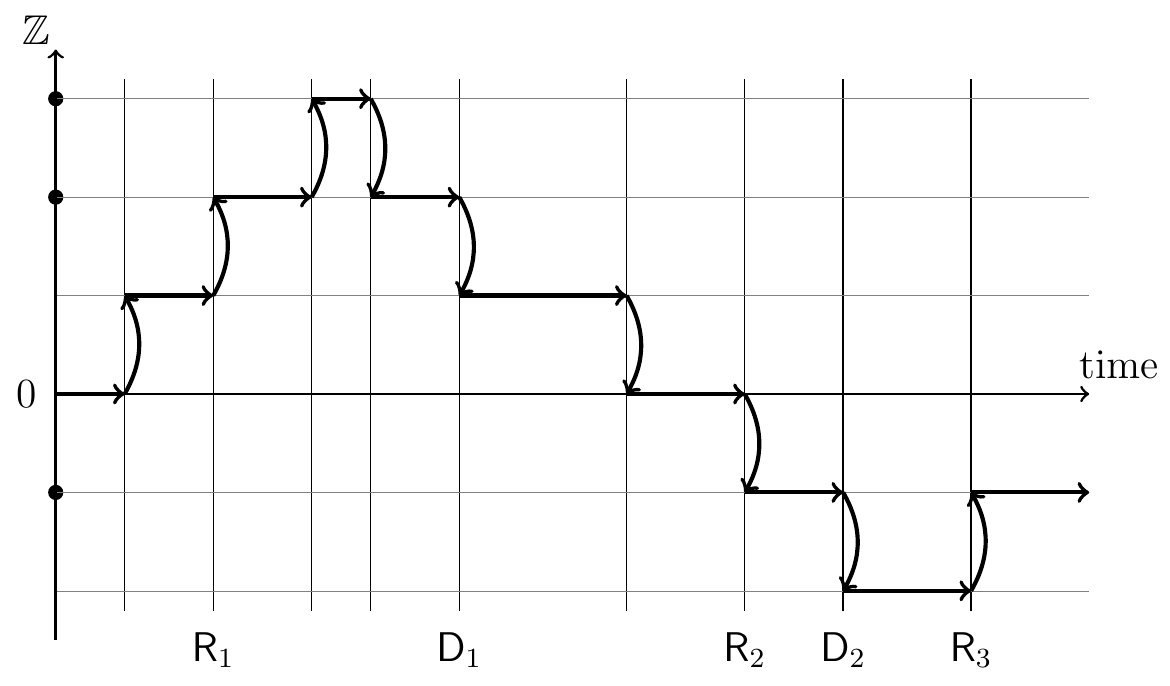} 
\caption{Schematic picture of returns/departure system. The arrows indicate the movements of the random walk and $\bullet$ are the points belonging to $\mathcal{H}_{1/2\alpha-\epsilon}$. \label{fig:RandD}}
\end{figure}
To this end, we need some controls on the structure of $\calH_{{1/2\alpha-\epsilon}}$. More precisely, for $\sfD_k-\sfR_k$ to be not too large, we need a control on the size of connected components of $\mathcal{H}_{1/2\alpha-\epsilon}$; and for $\sfR_k-\sfD_k$ to be not too large, we need to control the size of \emph{vacant} intervals. Here and below, the term interval indicates the set of the form $\{k+l\}_{1\le l \le m}$ for $k\in\Z$ and $m\in \N$, and we write it as $[k,k+m]$. The following lemma provides the above-mentioned controls. 
\begin{lemma}
\label{discrepancy}
When $\epsilon>0$ is sufficiently small, $\P$-almost surely, the following hold:
\begin{equation}
%
\text{any interval $I\subset\calH_{1/2\alpha-\epsilon}$ has length at most 3}
\label{cluster<4}
\end{equation}
and 
\begin{equation}
%
\text{any interval $I\subset[-t^{\frac12+\epsilon},t^{\frac12+\epsilon}]\setminus \calH_{1/2\alpha-\epsilon}$ has length at most $t^{\frac12-\epsilon}$}
\label{max-discrep}
\end{equation} 
for all sufficiently large $t$.
\end{lemma}
\begin{proof}
Using the assumption on the tail, we have 
\begin{equation}
\begin{split}
\P(x\in \calH_{1/2\alpha-\epsilon})
&\le \P(z(x)\ge t^{{1/2\alpha}-\epsilon})\\
&= t^{-{1/2+\alpha\epsilon}+o(1)}. 
\end{split}
\label{one-point}
\end{equation}
From this and the union bound, it is easy to see that 
\begin{equation}
\P\left(\exists \text{an interval }I \subset \mathcal{H}_{1/2\alpha-\epsilon}, |I| \ge 4\right)
\le t^{-\frac32+c\epsilon}
\end{equation}
and~\eqref{cluster<4} follows by the Borel--Cantelli lemma.

Next, for any $x\in [-t^{\frac12+\epsilon},t^{\frac12+\epsilon}]$, we have
\begin{equation}
\begin{split}
&\P\left(\exists y> x+t^{\frac12-\epsilon}\text{ such that }[x,y]\subset [-t^{\frac12+\epsilon},t^{\frac12+\epsilon}]\setminus \calH_{1/2\alpha-\epsilon}\right)\\
&\quad\le \left(1-t^{-{\frac{1}{2}+\alpha\epsilon}+o(1)}\right)^{t^{{1/2-}\epsilon}}\\
&\quad =\exp\set{-t^{{\epsilon(1-\alpha)}+o(1)}}
\end{split}
\end{equation}
by using~\eqref{one-point}, since the event in the first line requires that successive $t^{{1/2-\epsilon}}$ points fail to belong to $\mathcal{H}_{1/2\alpha-\epsilon}$. 
Since $\alpha<1$ and there are at most $2t^{1/2+\epsilon}$ choices of $x$ as the left endpoint of the interval $I$, the union bound and the Borel--Cantelli lemma yield~\eqref{max-discrep}.
\end{proof}
We use this lemma to derive an upper bound on $\sfR_1$ and 
$\sfR_{k+1}-\sfR_k=\sfR_2\circ \theta_{\sfR_k}$ for $k\ge 1$ 
(in the sense of stochastic domination). Since it is useful only when the 
random walk is started from the \emph{interior} 
\begin{equation}
\calH_{{1/2\alpha-\epsilon}}^\circ=\calH_{{1/2\alpha-\epsilon}}\setminus\{\inf\calH_{{1/2\alpha-\epsilon}},\sup\calH_{{1/2\alpha-\epsilon}}\}, 
\end{equation}
we introduce the stopping time 
\begin{equation}
 \sfT=\inf\set{u\ge 0\colon S_u\not \in (\inf\calH_{{1/2\alpha-\epsilon}},\sup\calH_{{1/2\alpha-\epsilon}})}
\end{equation}
and show the following lemma.
\begin{lemma}
\label{exit}
Fix $\epsilon>0$ and suppose that \eqref{max-discrep} holds. 
Then there exists $c>0$ such that
\begin{equation}
 P_0(\sfT\le t)
\le \exp\set{-ct^\epsilon}
\end{equation}
for all sufficiently large $t$. 
\end{lemma}
\begin{proof}
It follows from~\eqref{max-discrep} that
\begin{equation}
\left(-t^{{1+\epsilon\over 2}}, t^{{1+\epsilon\over 2}}\right)
\subset (\inf\calH_{{1/2\alpha-\epsilon}},\sup\calH_{{1/2\alpha-\epsilon}}). 
\end{equation}
Then the desired bound is a simple consequence of the reflection principle
and Lemma~\ref{lem:RWHK}.
\end{proof}
\begin{lemma}
\label{return}
Fix $\epsilon>0$ and suppose that~\eqref{cluster<4} and~\eqref{max-discrep} hold. Then there exist constants $c, C>0$ such that the following hold for all sufficiently large $t$:
\begin{align}
P_0\left(\sfR_1 \ge {\tfrac{t}{2}} \right)
&\le \exp\set{-c {t^\epsilon}},\label{hit-H}\\
\sup_{x\in \calH_{{1/2\alpha-\epsilon}}^\circ} 
P_x(\sfR_2 \ge r)
&\le {(Cr^{-{1\over 2}}\wedge 1)}1_{\{r\le t^{{1-{\epsilon}}}\}}
 +\exp\set{-c {r\over t^{{1-{2\epsilon}}}}}1_{\{r> t^{{1-{\epsilon}}}\}}.
\label{return-prob}
\end{align}
\end{lemma}
\begin{proof}
We first prove~\eqref{return-prob}. Let $x\in \calH_{{1/2\alpha-\epsilon}}^\circ$ and denote the left and right neighbors of $x$ in $\calH_{{1/2\alpha-\epsilon}}$ by $x_-$ and $x_+$ respectively, that is,
\begin{align}
x_-&=\max\{y\in\calH_{{1/2\alpha-\epsilon}}\colon y<x\},\\
x_+&=\min\{y\in\calH_{{1/2\alpha-\epsilon}}\colon y>x\}. 
\end{align}
Then by using the notion of hitting time $\sfH_y$ to 
a point $y\in\Z$, the return time $\sfR_2$ is written as
$\sfR_2=\sfD_1+(\sfH_x\wedge \sfH_{x_-}\wedge \sfH_{x_+})\circ \theta_{\sfD_1}$.
This allows us to bound the left-hand side of~\eqref{return-prob} as
\begin{equation}
\begin{split}
 P_x(\sfR_2 \ge r)
& \le P_x(\sfD_1 \ge r)
 +\max_{z\in \{x-1,\, x+1\}}P_z(\sfH_x\ge r)
1_{\{r\le t^{{1-\epsilon}}\}}\\
&\quad + P_x(\sfH_{x_-}\wedge \sfH_{x_+} \ge r)
1_{\{r> t^{{1-\epsilon}}\}}.
\end{split}
\end{equation}
By~\eqref{cluster<4}, the first term is bounded by $e^{-cr}$ and negligible compared with the right-hand side
of~\eqref{return-prob}. The probability in the second term is asymptotic to $(\pi r)^{-1/2}$ (see~\cite[Lemma~1 and (2.4) in Chapter III]{Fel68} for the discrete time analogue). 
The probability in the third term is for the random walk to stay in $(x_-,x_+)$ 
for a time duration $r$, which decays as 
\begin{equation}
P_x(\sfH_{x_-}\wedge \sfH_{x_+} \ge r)
\le \exp\left\{-c {r \over (x_+-x_-)^2}\right\}.
\label{stay-in-x+-}
\end{equation}
Combined with~\eqref{max-discrep}, this concludes the proof of~\eqref{return-prob}. 

Finally, the first assertion~\eqref{hit-H} follows in the same way as~\eqref{stay-in-x+-}. 
\end{proof}
\begin{proof}[Proof of Proposition~\ref{local-time}]
In view of Lemma~\ref{discrepancy}, we assume that~\eqref{max-discrep} holds
throughout the proof. 
Then Lemma~\ref{exit} yields that
\begin{equation}
\begin{split}
P_0\left(N_t< t^{{1\over 2}+\epsilon}\right)
&=P_0\left(\sfR_{\lfloor t^{{1/ 2}+\epsilon}\rfloor+1} \ge t\right)\\
&\le P_0\left(\sfR_{\lfloor t^{{1/2}+\epsilon}\rfloor+1}\wedge\sfT \ge t\right)
 +\exp\{-ct^\epsilon\}.
\end{split}
\end{equation}
Recalling $\sfR_0=0$, we rewrite the first term as
\begin{equation}
\begin{split}
&P_0\left(\sfR_{\lfloor t^{{1/2}+\epsilon}\rfloor+1}\wedge\sfT \ge t\right)\\
&\quad = P_0\left(\sum_{k\le {t^{{1/2}+\epsilon}}}(\sfR_{{k+1}}\wedge\sfT-\sfR_{{k}}\wedge\sfT) \ge t\right)\\
&\quad {= P_0\left(\sfR_1 \ge \tfrac{t}{2} \right)
+P_0\left(\sum_{1\le k\le t^{{1/2}+\epsilon}}(\sfR_{k+1}\wedge\sfT-\sfR_k\wedge\sfT) \ge \tfrac{t}{2}\right)}.
\end{split}
\end{equation}
The first term is bounded by $\exp\{-ct^{\epsilon}\}$ by~\eqref{hit-H}. By the strong Markov property and Lemma~\ref{return}, the summands in the right-hand side are stochastically dominated by a family of independent and 
identically distributed random variables $\{\mu_k\}_{k< t^{{1/2}+\epsilon}+1}$
whose tail distribution is given by the right-hand side of~\eqref{return-prob}. 
We assume that this family is defined on the same probability space as the 
random walk with the measure $P_0$. 
Then, the above right-hand side is bounded by
\begin{equation}
\begin{split}
&P_0\left(\sum_{{1\le k\le t^{{1/2}+\epsilon}}}\mu_k \ge t\right)\\
&\quad\le P_0\left(\sum_{{1\le k\le t^{{1/2}+\epsilon}}}\mu_k{\cdot 1_{\{\mu_k \le t^{1-\epsilon}\}}}\ge t\right)
+P_0\left(\max_{{1\le k\le t^{{1/2}+\epsilon}}}\mu_k > t^{1-\epsilon}\right)\\
&\quad\le P_0\left(\sum_{{1\le k\le t^{{1/2}+\epsilon}}}\mu_k{\cdot 1_{\{\mu_k \le t^{1-\epsilon}\}}}\ge t\right)+3t^{{1\over 2}+\epsilon}\exp\{-c t^\epsilon\}.
\end{split}
\end{equation} 
In order to bound the first term, we use a large deviation principle for truncated sums proved in~\cite{Cha12}. We summarize the statement for the reader's convenience. Let $(\{H_k\}_{k\in\N}, \mathbf{P})$ be $\R^d$-valued independent and identically distributed random variables with \emph{power law tail} (see~\cite[(1.1)]{Cha12} for the definition), and let $M_n>0$ be a sequence satisfying $\lim_{n\to\infty}n\mathbf{P}(|H_1| > M_n)=\infty$. Then the truncated sum
\begin{equation}
 \frac{1}{nM_n\mathbf{P}(|H_1|>M_n)}\sum_{k=1}^n H_k\cdot 1_{\{|H_k| \le M_n\}}
\end{equation}
satisfies a large deviation principle with speed $n\mathbf{P}(|H_1|>M_n)$ whose rate function vanishes only at zero. We apply this theorem with the choice
\begin{equation}
 n=\lfloor t^{\frac12+\epsilon}\rfloor,\ M_n=t^{1-\epsilon}, \text{ and } \mathbf{P}(H_k>r)=Cr^{-{1\over 2}}\wedge 1,
\end{equation}
so that $H_k\cdot 1_{\{|H_k|\le M_n\}}$ has the same law as $\mu_k\cdot 1_{\{\mu_k \le t^{1-\epsilon}\}}$. Then~\cite[Theorem~3.1]{Cha12} yields 
\begin{equation}
 P_0\left(\sum_{1\le k\le t^{{1/2}+\epsilon}}\mu_k{\cdot 1_{\{\mu_k \le t^{1-\epsilon}\}}}\ge t\right)
\le \exp\{-c t^\epsilon\}.
\end{equation}
Therefore we conclude that when $\epsilon>0$ is small, 
\begin{equation}
P_0\left(N_t< t^{{1\over 2}+\epsilon}\right)
\le \exp\set{-ct^{\epsilon}} 
\end{equation}
for all sufficiently large $t$. 

Finally, on the event $\{N_t\ge t^{1/2+\epsilon}\}$, we have 
\begin{equation}
 \ell_t(\calH_{{1/2\alpha-\epsilon}}) \ge \sum_{k\le t^{1/2+\epsilon}-1}(\sfD_k-\sfR_k),
\end{equation} 
whose right-hand side is bounded from below by a sum of the independent exponential random variables with rate one. 
Therefore, a simple large deviation bound leads us to
\begin{equation}
\begin{split}
&P_0\left(\ell_t(\calH_{{1/2\alpha-\epsilon}}) \le t^{{1\over 2}-\epsilon}\right)\\
&\quad\le P_0\left(N_t < t^{{1\over 2}+\epsilon}\right)
+P_0\left(\sum_{k\le t^{1/2+\epsilon}-1}(\sfD_k-\sfR_k)\le t^{{1\over 2}-\epsilon}\right)\\
&\quad\le \exp\set{-ct^\epsilon}
\end{split}
\end{equation}
and we are done. 
\end{proof}
\subsection{Upper bound for $\alpha<1$ and $d\ge 2$}
In this case, we have $s(d,\alpha)=1/\alpha$. 
To prove Proposition~\ref{prop:RWRSlower} in this case, we introduce the \emph{relevant} level set:
\begin{equation}
\calH_{{*}}=\set{x\in \Z^d \colon z(x)\ge t^{{1\over \alpha}-\epsilon}}.
\end{equation}
This is effectively the same as $\calH_{1/\alpha-\epsilon}$ used before, but we will take intersection with $[-t,t]^d$ instead of $[-t^{1/2+\epsilon},t^{1/2+\epsilon}]^d$ to simplify the notation below.
Then it suffice to prove the following proposition: 
\begin{proposition}
 \label{local-time3} 
Let $d \ge 2$. For sufficiently small $\epsilon>0$ depending on $\alpha$, there exists $c(\epsilon)>0$ such that $\P$-almost surely, 
\begin{equation}
 P_0\left(\ell_t(\calH_{{*}})\le 1\right)
\le \exp\{-t^{c(\epsilon)}\}
\end{equation}
for all sufficiently large $t$. 
\end{proposition}
We first estimate the probability that the random walk avoids $\calH_{{*}}$ while staying inside a large box $[-t,t]^d$. Let us denote the exit time from this box by
\begin{equation}
 \sfT_{{*}}=\inf\left\{u\ge 0\colon S_u \not\in \left[-t, t\right]^d \right\}.
\end{equation}
\begin{lemma}
\label{lem:avoid-H}
Let $d\ge 2$. For sufficiently small $\epsilon>0$ depending on $\alpha$, there exists $c(\epsilon)>0$ such that $\P$-almost surely, 
\begin{equation}
\max_{x\in[-t, t]^d} P_x\left(\sfH_{\calH_{{*}}}\wedge \sfT_{{*}}>t^{1-\frac{\alpha\epsilon}{8d}}\right)
\le \exp\left\{-t^{c(\epsilon)}\right\}
\end{equation}
for all sufficiently large $t$.
\end{lemma}
\begin{proof}
We use the ``method of enlargement of obstacles'' in~\cite{Szn98} with a slight modification made in~\cite{Fuk09a}. 
Strictly speaking, the method in~\cite{Szn98,Fuk09a} is developed in the continuum setting but it can be adapted to the discrete setting as is done in~\cite{BAR00}. 

Let us choose the parameters in~\cite{Fuk09a}. First, let us define the scale $r=t^{1/2-\delta}$ with $\delta=\alpha\epsilon/(4+d)$ and divide $[-t,t]^d$ into the boxes of the form $r(q+[0,1)^d)$ ($q\in \Z^d$). Then we scale the space by $r^{-1}$ so that we have the box $[-t/r,t/r]^d$ divided into unit cubes. Next, let  $\gamma=(d-2)/d+\delta$ and divide the unit cubes into dyadic boxes (which we call \emph{mesoscopic cells}) with side length $2^{-n_\gamma}\in [r^{-\gamma},2r^{-\gamma})$. 
\begin{lemma}
\label{lem:EveryBox}
Let $d\ge 2$. Then $\P$-almost surely, for all sufficiently large $t$, 
every $2^{-n_\gamma-1}(q+[0,1]^d)$ intersecting $[-t/r,t/r]^d$ contains a point of $r^{-1}\calH_{{*}}$. 
\end{lemma}
\begin{proof}
Observe that every mesoscopic cell $2^{-n_\gamma-1}(q+[0,1]^d)$ ($q\in \Z^d$) contains at least $2^{-d}r^{d(1-\gamma)}\ge t^{1-(4+d)\delta/2}$ points of $r^{-1}\Z^d$ for sufficiently large $t$. Since 
\begin{equation}
\P(x\in \calH_{{*}})=\P\left(z(x)\ge t^{\frac{1}{\alpha}-\epsilon}\right)
=t^{-1+\alpha\epsilon} \textrm{ as }t\to\infty, 
\end{equation}
it follows that 
\begin{equation}
\begin{split}
\P\left((q+[0,1)^d) \cap r^{-1}\calH_{{*}}=\emptyset\right)
&\le (1-t^{-1+\alpha\epsilon})^{t^{1-(4+d)\delta/2}}\\
&\le \exp\{-t^{\alpha\epsilon/2}\}
\end{split}
\end{equation}
for all sufficiently large $t$. Since there are only polynomially many mesoscopic cells intersecting $[-t/r,t/r]^d$, the assertion follows by the union bound and the Borel--Cantelli lemma. 
\end{proof}
If all the mesoscopic cells in a unit cube $q+[0,1)^d$ contains a point of $r^{-1}\calH_{{*}}$, then since the volume of each cell ($\asymp r^{-(d-2)-d\delta}$) is much smaller than the capacity of a point in the scaled lattice ($\asymp r^{-(d-2)}$), the unit cube $q+[0,1)^d$ is more crowded than the so-called ``constant capacity regime'' in the crushed ice problem, see~\cite[p.116]{Szn98}. Roughly speaking, the method of enlargement of obstacles in~\cite{Szn98} allows us to \emph{solidify} such a unit cube when we consider the principal eigenvalue $\lambda_1(U)$ of $-\tfrac{1}{2d}\Delta$ in $U\subset\Z^d$ with the Dirichlet boundary condition. More precisely, such a cube is contained in the \emph{density} set $\mathcal{D}_r(\calH_{{*}})$ defined in~\cite[p.152]{Szn98} and hence it follows from~\cite[Theorem~2.3 on p.158]{Szn98} that there exists $\rho>0$ such that for all $M>0$ and sufficiently large $t$, 
\begin{equation}
\label{eq:spectral}
\left|\lambda_1([-t/r,t/r]^d\setminus r^{-1}\calH_{{*}})\wedge M-\lambda_1([-t/r,t/r]^d\setminus {\mathcal{D}}_r(\calH_{{*}}))\wedge M\right|\le r^{-\rho}.
\end{equation}

In our setting, Lemma~\ref{lem:EveryBox} and~\cite[Proposition~2.7]{Fuk09a} imply that every $q+[0,1)^d$ intersecting $[-t/r,t/r]^d$ belongs to $\mathcal{D}_r(\calH_{{*}})$ and hence $[-t/r,t/r]^d\setminus \mathcal{D}_r(\calH_{{*}})=\emptyset$. Therefore~\eqref{eq:spectral} implies that $\lambda_1([-t/r,t/r]^d\setminus r^{-1}\calH_{{*}})\wedge M$ is arbitrarily close to $M>0$. Reverting the scaling, we conclude that for any $M>0$, $\P$-almost surely, 
\begin{equation}
\begin{split}
  \lambda_1([-t,t]^d\setminus \calH_{{*}})&\ge Mr^{-2}\\
&=Mt^{-1+\frac{2\alpha\epsilon}{4+d}}
\end{split}
\end{equation}
for all sufficiently large $t$. By using this eigenvalue bound together with a standard semigroup bound~\cite[(eq.~(2.21))]{Kon16}, we obtain
\begin{equation}
\label{eq:avoid-H2}
\begin{split}
\sup_{x\in[-t,t]^d}P_x\left(\sfH_{\calH_{{*}}}\wedge \sfT_{{*}}>t^{1-\frac{\alpha\epsilon}{8d}}\right)
& \le (2t)^{d/2}\exp\left\{-t^{1-\frac{\alpha\epsilon}{8d}}\lambda_1([-t,t]^d\setminus \calH_{{*}})\right\}\\
& \le \exp\left\{-t^{c(\epsilon)}\right\}.
\end{split}
\end{equation}
This completes the proof of Lemma~\ref{lem:avoid-H}.
\end{proof}
We also need the following sparsity result which can be proved in the same way as~\eqref{cluster<4}.
\begin{lemma}
\label{lem:cluster}
For $d\ge 2$, $\P$-almost surely, 
every connected component of $\calH_* \cap [-t,t]^d$ contains at most $d+2$ points for all sufficiently large $t$.  
\end{lemma}

Let us define the times of returns to/departures from $\calH_{{*}}$: $\sfR_0=\sfD_0=0$ and for $k\ge 1$,  
\begin{align}
\sfR_k&=\inf\set{u\ge \sfD_{k-1}\colon S_u\in\calH_{{*}}},\\
\sfD_k&=\inf\set{u\ge \sfR_k\colon S_u\not\in\calH_{{*}}}.
\end{align}
We have the following bound on the number of returns before time $t$ defined by\begin{equation}
 N_t=\sup\set{k\ge 1\colon \sfR_k <t}.
\end{equation}
\begin{lemma}
\label{number-hit-H}
Let $d\ge 2$. Then there exists $c(\epsilon)>0$ such that $\P$-almost surely, 
for all sufficiently large $t$, 
\begin{equation}
 P_0\left(N_t\le t^{{\alpha\epsilon \over 16d}}\right)
\le \exp\left\{-t^{c(\epsilon)}\right\}.
\end{equation}
\end{lemma}
\begin{proof}
By the reflection principle and Lemma~\ref{lem:RWHK}, it follows that
\begin{equation}
 P_0\left(\sfT_{{*}} \le t\right)\le \exp\left\{-ct\right\}
\end{equation}
for some $c>0$. Thus it suffices to show that 
\begin{equation}
 P_0\left(N_t\le t^{{\alpha\epsilon \over 16d}}, \sfT_{{*}} >t\right)
\le \exp\left\{-t^{c(\epsilon)}\right\}. 
\end{equation}
Observe that if $N_t\le t^{{\alpha\epsilon \over 16d}}$ and $\sfT_{{*}}>t$, then there exists $k\le t^{{\alpha\epsilon \over 16d}}+1$ such that 
\begin{equation}
 \sfR_k\wedge \sfT_{{*}}-\sfR_{k-1}\wedge \sfT_{{*}}>t^{1-{\alpha\epsilon \over 16d}},
\end{equation}
and this event further implies that
\begin{equation}
 \textrm{either }\sfR_k\wedge \sfT_{{*}}-\sfD_{k-1}\wedge \sfT_{{*}}
\ge t^{1-{\alpha\epsilon \over 8d}} 
\textrm{ or } 
\sfD_{k-1}\wedge \sfT_{{*}}-\sfR_{k-1}\wedge \sfT_{{*}}\ge t^{1-{\alpha\epsilon \over 8d}}. 
\end{equation}
Then by the union bound and the strong Markov property, we obtain
\begin{equation}
\begin{split}
&P_0\left(N_t\le t^{{\alpha\epsilon \over 16d}}, \sfT_{{*}} >t\right)\\
&\quad\le \sum_{k\le t^{{\alpha\epsilon / 16d}}+1}
 P_0\left(\sfD_{k-1}\wedge \sfT_{{*}}-\sfR_{k-1}\wedge \sfT_{{*}}>t^{1-{\alpha\epsilon \over 8d}}\right)\\
&\quad\qquad+\sum_{k\le t^{{\alpha\epsilon / 8d}}+1}
 P_0\left( \sfR_k\wedge \sfT_{{*}}-\sfD_{k-1}\wedge \sfT_{{*}}>t^{1-{\alpha\epsilon \over 8d}}\right)\\
&\quad\le \left(t^{{\alpha\epsilon \over 16d}}+1\right)
\left(
\max_{x\in\calH_{{*}}}P_x\left(\sfD_1>t^{1-{\alpha\epsilon \over 8d}}\right)
+\max_{x\in[-t, t]^d} 
P_x\left(\sfH_{\calH_{{*}}}\wedge \sfT_{{*}}>t^{1-{\alpha\epsilon \over 8d}}\right)\right).
\end{split}
\end{equation}
By Lemma~\ref{lem:cluster}, $D_1$ under $P_x$ for $x\in\calH_{{*}}$ has an exponential tail and hence the first probability in the last line decays exponentially in $t^{1-\alpha\epsilon/8d}$. The second probability is bounded by $\exp\{-t^{c(\epsilon)}\}$ by~\eqref{eq:avoid-H2}. 
\end{proof}
\begin{proof}[Proof of Proposition~\ref{local-time3}]
On the event $\{N_t\ge t^{\alpha\epsilon / 8d}\}$, we have 
\begin{equation}
 \ell_t(\calH_{{*}}) \ge \sum_{k\le t^{\alpha\epsilon/ 8d}-1}(\sfD_k-\sfR_k),
\end{equation} 
whose right-hand side is bounded from below by a sum of the independent exponential random variables with rate one. 
Therefore, it follows from Lemma~\ref{number-hit-H} that 
\begin{equation}
\begin{split}
P_0(\ell_t(\calH_{{*}}) \le 1)
&\le P_0\left(N_t < t^{\alpha\epsilon \over 8d}\right)
+P_0\left(\max_{k\le t^{\alpha\epsilon/ 8d}-1}(\sfD_k-\sfR_k)\le 1\right)\\
&\le \exp\set{-t^{c(\epsilon)}}+\exp\left\{-t^{\alpha\epsilon \over 8d}+1\right\}
\end{split}
\end{equation}
and we are done. 
\end{proof}
\section{Local central limit theorem or failure of thereof}
\label{sec:LCLT}
In this section, we prove that the local central limit theorem holds when $\alpha>\tfrac{d}2\vee 1$, and fails when $\alpha<\tfrac{d}{2}\vee 1$. For the latter, since the case $\alpha<1$ is covered by Theorem~\ref{thm:on-diag}, we focus on $d\ge 3$ and $1\le \alpha<\tfrac{d}{2}$. For simplicity, we set $\E[z(0)]=1$ so that $P^{\E[\omega]}_0=P_0$. 
\begin{proof}
[Proof of Proposition~\ref{rem:LCLT}]
Let us first fix $\alpha>\tfrac{d}{2}\vee 1$, $\epsilon>0$, $R>0$, and $(x_1,x_2)\in\Z^{1+d}$ with $|(x_1,x_2)|\le Rt^{1/2}$ and consider the transition probability
\begin{equation}
\begin{split}
P^\omega_0(X_t=(x_1,x_2))
&= E_0\left[p_{A(t)}(0,x_1) ; S_t=x_2\right]\\
&= E_0\left[p_{A(t)}(0,x_1) ; S_t=x_2, |t^{-1}A(t)-1|\le \epsilon\right]\\
&\quad+E_0\left[p_{A(t)}(0,x_1) ; S_t=x_2, |t^{-1}A(t)-1|> \epsilon\right].
\end{split}
\label{eq:lclt1}
\end{equation}
For the second term on the right-hand side, we use Theorem~\ref{LDP} and Proposition~\ref{prop:RWRSlower} to obtain
\begin{equation}
 P_0\left(S_t=x_2, |t^{-1}A(t)-1|> \epsilon \right)\le \exp\left\{-t^{c(\epsilon)}\right\},
\label{eq:LD_pinned}
\end{equation}
which is negligible for the purpose of the local central limit theorem. 
Next, for the first term on the right-hand side of~\eqref{eq:lclt1}, note first that 
\begin{equation}
 \left|\frac{p_{A(t)}(0,x_1)}{p_t(0,x_1)}-1\right| \le c\epsilon
\end{equation}
on the event $\{|t^{1}A(t)-1|\le \epsilon\}$ for all sufficiently large $t$. Moreover, by using~\eqref{eq:LD_pinned}, we have that
\begin{equation}
P_0\left(S_t=x_2, |t^{-1}A(t)-1|\le \epsilon\right) 
\sim P_0\left(S_t=x_2\right)
\end{equation}
as $t\to\infty$. Substituting these controls into~\eqref{eq:lclt1}, we arrive at
\begin{equation}
\left|\frac{P^\omega_0(X_t=(x_1,x_2))}{P_0\left((S^1_t,S^2_t)=(x_1,x_2)\right)}
-1\right|\le c\epsilon
\end{equation}
for all sufficiently large $t$. Since $\epsilon>0$ was arbitrary, we are done. 

Next, let $d \ge 3$ and $1\le \alpha<\tfrac{d}{2}$. Then for any $R>0$ and $\epsilon\in (0,1)$, there is an $x_2\in [-Rt^{1/2}, Rt^{1/2}]$ such that $z(x_2)\ge t^{\frac{d}{2\alpha}-\epsilon}$ almost surely for all large $t$. For this $x_2$, by time reversal, we have 
\begin{equation}
\begin{split}
P^\omega_0(X_t=(x_1,x_2))
&= E_0\left[p_{A(t)}(0,x_1) ; S_t=x_2\right]\\
&= E_{x_2}\left[p_{A(t)}(x_1,0) ; S_t=0, \sfD_1 < t^{-\epsilon}\right]\\
&\quad+E_{x_2}\left[p_{A(t)}(x_1,0) ; S_t=0, \sfD_1 \ge t^{-\epsilon}\right],
\end{split}
\label{eq:first-jump}
\end{equation}
where $\sfD_1$ is the first jump time of the random walk. 
By Proposition~\ref{prop:RWRSlower} and Lemma~\ref{lem:RWHK}, the first term on the right-hand side is bounded by 
\begin{equation}
\begin{split}
&E_{x_2}\left[p_{A(t)}(0,x_1) ; S_t=0, A(t) \ge t^{1-\epsilon}, \sfD_1 \le t^{-\epsilon}\right]+\exp\{-t^{c(\epsilon)}\}\\
&\quad \le ct^{-\frac{1-\epsilon}{2}}P_{x_2}\left(S_t=0, \sfD_1 \le t^{-\epsilon}\right)+\exp\{-t^{c(\epsilon)}\}.
\end{split}
\end{equation}
By using the strong Markov property at $\sfD_1$ and Lemma~\ref{lem:RWHK} again, we get
\begin{equation}
 P_{x_2}\left(S_t=0, \sfD_1 \le t^{-\epsilon}\right)\le ct^{-\epsilon-\frac{d}{2}}
\end{equation}
and hence the first term on the right-hand side of~\eqref{eq:first-jump} is bounded by $ct^{-\frac{1+d}{2}-\frac{\epsilon}{2}}$. 
For the second term in~\eqref{eq:first-jump}, note that $A(t)\ge t^{\frac{d}{2\alpha}-2\epsilon}$ on the event $\{\sfD_1 \ge t^{-\epsilon}\}$ since $z(S_0)=z(x_2)\ge t^{\frac{d}{2\alpha}-\epsilon}$. Then by using Lemma~\ref{lem:RWHK} twice, we have
\begin{equation}
\begin{split}
 E_{x_2}\left[p_{A(t)}(x_1,0) ; S_t=0, \sfD_1 \ge t^{-\epsilon}\right]
 &\le c t^{-\frac{d}{4\alpha}+\epsilon}P_{x_2}\left(S_t=0\right)\\
 &\le ct^{-\frac{d}{4\alpha}+\epsilon-\frac{d}{2}}.
\end{split}
\end{equation}
This is bounded by $t^{-\frac{1+d}{2}-\frac{\epsilon}{2}}$ for sufficiently small $\epsilon>0$ since $\alpha < \tfrac{d}{2}$. Therefore we conclude that
\begin{equation}
 P^\omega_0(X_t=(x_1,x_2))\le ct^{-\frac{1+d}{2}-\frac{\epsilon}{2}}.
\label{eq:LCLTfail}
\end{equation}
\end{proof}
\section{Power law decay regime of random conductance model}
\label{sec:RCM-power}
\begin{proof}[Proof of Theorem~\ref{thm:RCM-power}]
We consider $\delta>\frac{1}{2\alpha}\vee \frac12$ first. By~\eqref{HK-RWRS} and Lemma~\ref{lem:RWHK}, 
\begin{equation}
 \begin{split}
\label{eq:Gaussian}
 P^\omega_0(X_t=t^\delta\mathbf{e}_1)
 &=E_0[p_{A(t)}(0,t^\delta); S_t=0]\\
 &\asymp E_0\left[A(t)^{-1/2}\exp\left\{-\frac{t^{2\delta}}{2A(t)}\right\}; A(t) \ge t^{\delta}, S_t=0\right]\\
&\qquad +E_0\left[\exp\left\{-t^{\delta}\left(1\vee \log \frac{t^\delta}{A(t)}\right)\right\}; A(t)<t^\delta, S_t=0\right].
 \end{split}
\end{equation}
The second term on the right-hand side is stretched exponentially small and hence negligible. Let $\epsilon>0$ be a small constant such that $2\delta-\epsilon>\frac{1}{\alpha}\vee 1$. If $A(t)\le t^{2\delta-\epsilon}$, then the exponential factor in the first term decays stretched exponentially and hence we can drop this event. In this way, we arrive at the upper bound 
\begin{equation}
\begin{split}
 P^\omega_0(X_t=t^\delta\mathbf{e}_1)
 & \le c E_0\left[A(t)^{-1/2}; A(t)\ge t^{2\delta-\epsilon}, S_t=0\right]\\
 &= c\int_{t^{2\delta-\epsilon}}^\infty u^{-1/2}P_0(A(t)\in \text{d} u, S_t=0)\\
 &= ct^{-\delta+\frac{1}{2}\epsilon}P_0\left(A(t)\ge t^{2\delta-\epsilon}, S_t=0\right)\\
&\qquad +\frac{c}{2} \int_{t^{2\delta-\epsilon}}^\infty u^{-3/2}P_0(A(t)\ge u, S_t=0)\dd u,
\end{split}
\label{eq:IBP}
\end{equation}
where we have used the integration by parts in the last equality. We can also obtain a similar lower bound on $P^\omega_0(X_t=t^\delta\mathbf{e}_1)$ with $t^{2\delta-\epsilon}$ replaced by $t^{2\delta+\epsilon}$. Since the following argument is insensitive to this change, we will only estimate the right-hand side of~\eqref{eq:IBP}. The first term gives the desired asymptotics by~\eqref{eq:lbd_pinned}. As for the second term, the contribution from the region $u>t^{\frac{d}{2\alpha}-\epsilon}$ is seen to be negligible by Theorem~\ref{RWRS} (for $u>t^{\frac{d}{2\alpha}+\epsilon}$) and the upper bound in Theorem~\ref{power-law} and Remark~\ref{rem:coverage} (for $t^{\frac{d}{2\alpha}-\epsilon}<u\le t^{\frac{d}{2\alpha}+\epsilon}$). 
Hence the relevant region is $u\in [t^{2\delta-\epsilon},t^{\frac{d}{2\alpha}-\epsilon}]$: 
\begin{equation}
\int_{t^{2\delta-\epsilon}}^\infty u^{-3/2}P_0\left(A(t)\ge u, S_t=0\right)\dd u
\sim \int_{t^{2\delta-\epsilon}}^{t^{d/(2\alpha)-\epsilon}} u^{-3/2}P_0\left(A(t)\ge u, S_t=0\right)\dd u
\label{eq:relevant}
\end{equation}
as $t\to\infty$. We divide the above integral into those on intervals of the form $[t^{k\epsilon},t^{(k+1)\epsilon}]$ and apply Theorem~\ref{power-law} to obtain 
\begin{equation}
\begin{split}
&\int_{t^{2\delta-\epsilon}}^{t^{d/(2\alpha)-\epsilon}} u^{-3/2}P_0\left(A(t)\ge u, S_t=0\right)\dd u\\
&\quad \le \sum_{k=\lfloor 2\delta/\epsilon\rfloor-1}^{\lfloor d/(2\alpha\epsilon)\rfloor}t^{-\frac{3k\epsilon}{2}}P_0\left(A(t)\ge t^{k\epsilon}, S_t=0\right)(t^{(k+1)\epsilon}-t^{k\epsilon})\\
&\quad \le \sum_{k=\lfloor 2\delta/\epsilon\rfloor-1}^{\lfloor d/(2\alpha\epsilon)\rfloor}t^{-\frac{3k\epsilon}{2}} t^{-\alpha k\epsilon+1-\frac{d}{2}+o(1)} (t^{(k+1)\epsilon}-t^{k\epsilon})\\
&\quad \le t^{c\epsilon} \int_{t^{2\delta-\epsilon}}^{t^{d/(2\alpha)-\epsilon}} u^{-\frac32-\alpha} t^{1-\frac{d}{2}+o(1)} \dd u\\
&\quad \le t^{-\delta(1+2\alpha)+1-\frac{d}{2}+c'\epsilon}
\end{split}
\label{discretization}
\end{equation}
as $t \to \infty$. As $\epsilon>0$ is arbitrary, this yields the desired upper bound. We can get the lower bound in a similar way and complete the proof in the case $\delta>\frac{1}{2\alpha}\vee \frac{1}{2}$.

Next we consider $\delta\le \frac{1}{2\alpha}\vee \frac12=\frac{s(d,\alpha)}{2}$. In this case, the interval $[t^{2\delta},\infty)$ in~\eqref{eq:IBP} contains the typical scale $t^{s(d,\alpha)}$ of $A(t)$. For the upper bound, we may drop the exponential factors in~\eqref{eq:Gaussian}. Since the lower deviation event $\{A(t)< t^{s(d,\alpha)-\epsilon}\}$ is stretched exponentially unlikely by Proposition~\ref{prop:RWRSlower}, we can discard it to obtain
\begin{equation}
\begin{split}
P^\omega_0(X_t=t^\delta\mathbf{e}_1)
& \le (1+o(1)) E_0\left[A(t)^{-1/2}; A(t)\ge t^{s(d,\alpha)-\epsilon}, S_t=0\right]\\
& =t^{-\frac{s(d,\alpha)}{2}+\frac{1}{2}\epsilon}P_0\left(A(t)\ge t^{s(d,\alpha)-\epsilon}, S_t=0\right)\\
&\qquad +\frac12 \int_{t^{s(d,\alpha)-\epsilon}}^\infty u^{-3/2}P_0(A(t)\ge u, S_t=0)\dd u
\end{split}
\label{eq:2nd-1}
\end{equation}
as $t\to \infty$ as before. The first term gives the desired asymptotics since $P_0(S_t=0)\asymp t^{-d/2}$ and $P_0(A(t)<t^{s(d,\alpha)-\epsilon})$ decays stretched exponentially. As for the second term, note first that 
\begin{equation}
\begin{split}
& \int_{t^{s(d,\alpha)-\epsilon}}^{t^{s(d,\alpha)+\epsilon}} u^{-\frac32}P_0(A(t)\ge u, S_t=0)\dd u \\
&\quad \le t^{-\frac{s(d,\alpha)}{2}+\frac52\epsilon}P_0\left(A(t)\ge t^{s(d,\alpha)-\epsilon}, S_t=0\right)
\end{split}
\end{equation}
for sufficiently large $t$ and the right-hand side is almost the same order as the first term. Thus it remains to show that the integral over $[t^{s(d,\alpha)+\epsilon},\infty)$ is negligible. Using Theorem~\ref{power-law} and the same argument as in~\eqref{discretization}, we can show that
\begin{equation}
\begin{split}
&\int_{t^{s(d,\alpha)+\epsilon}}^\infty u^{-\frac32}P_0(A(t)\ge u, S_t=0)\dd u\\
&\quad\le t^{c\epsilon}\int_{t^{s(d,\alpha)+\epsilon}}^\infty u^{-\frac32-\alpha}t^{1-\frac{d}{2}}\dd u\\
&\quad\le t^{-\frac{s(d,\alpha)}{2}(1+2\alpha)+1-\frac{d}{2}+c'\epsilon}
\end{split}
\label{eq:2nd-2}
\end{equation}
for sufficiently large $t$. The last line is of small order compared with the desired asymptotics and the upper bound follows. 
To show the lower bound, note that~\eqref{eq:lbd_pinned} implies
\begin{equation}
 P_0\left(A(t) \in \left[t^{s(d,\alpha)},t^{s(d,\alpha)+\epsilon}\right], S_t=0\right) 
\ge t^{-c\epsilon}P_0\left(S_t=0\right)
\end{equation}
for all sufficiently large $t$. Thus we can replace the event $\{S_t=0\}$ in~\eqref{eq:Gaussian} by the one on the above left-hand side to obtain
\begin{equation}
\begin{split}
P^\omega_0(X_t=t^{\delta}\mathbf{e}_1)
& \ge E_0\left[A(t)^{-1/2}; A(t) \in \left[t^{s(d,\alpha)},t^{s(d,\alpha)+\epsilon}\right],S_t=0\right]\\
& \ge c t^{-\frac{s(d,\alpha)}{2}-c'\epsilon}P_0\left(S_t=0\right),
\end{split}
\end{equation}
which yields the desired lower bound.
\end{proof}

\section{Asymptotics of the Green function}
\label{sec:Green}
In this section, we prove Theorem~\ref{thm:Green}. Essentially, it is a consequence of the heat kernel estimates that have been proved so far. However, since we only know those estimates in the long time asymptotics, we need an a priori bound to deal with the first small time interval in the integral~\eqref{eq:def_Green}.

\begin{proof}[Proof of Theorem~\ref{thm:Green}]
Let $\epsilon>0$ and define 
\begin{equation}
c(\epsilon)=
\begin{cases}
\epsilon, &d=1,\\
\frac{4\alpha}{d}-2\epsilon, &d\ge 2. 
\end{cases}
\end{equation}
We start by proving that we may remove the first $[0,n^{c(\epsilon)}]$ interval from~\eqref{eq:def_Green} to show
\begin{equation}
g^\omega(0,n\mathbf{e}_1)
\sim \int_{n^{c(\epsilon)}}^\infty P^\omega_0(X_t=n\mathbf{e}_1) \dd t
\label{eq:int}
\end{equation}
as $n\to\infty$. To this end, we recall the representation $X_t=(S^1_{A^2(t)},S^2_t)$ and argue as
\begin{equation}
\begin{split}
P^\omega_0(X_t=n\mathbf{e}_1)&\le P_0(A(t)\ge n^{2-\epsilon})+\sup_{s\le n^{2-\epsilon}}p_s(0,n)\\
&\le P_0\left(A(n^{c(\epsilon)})\ge n^{2-\epsilon}\right)+\sup_{s\le n^{2-\epsilon}}p_s(0,n)
\end{split}
\end{equation}
for any $t\le n^{c(\epsilon)}$. Now by Theorem~\ref{RWRS}, the first term on the right-hand side decays stretched exponentially. The second term also decays stretched exponentially in $n$ by Lemma~\ref{lem:RWHK}. 

The cases $d\le 2$ and $d\ge 3$ with $\alpha\ge\frac{d}{2}$ are fairly simple. For any $\epsilon>0$, the event $\{A(t)\ge t^{s(d,\alpha)+\epsilon}\}\cup\{A(t)\le t^{s(d,\alpha)-\epsilon}\}$ has stretched exponentially small probability. Hence we can approximate the Green function as 
\begin{equation}
\begin{split}
g^\omega(0,n\mathbf{e}_1)&\sim\int_{n^{c(\epsilon)}}^\infty {(4\pi t)}^{-\frac{d}{2}}p_{t^{s(d,\alpha)+o(1)}}(0,n)\dd t\\
&\sim \int_{\epsilon n^{2/s(d,\alpha)}}^\infty {(4\pi t)}^{-\frac{d}{2}}p_{t^{s(d,\alpha)+o(1)}}(0,n)\dd t
\end{split}
\end{equation}
as $n\to\infty$ and the desired bound readily follows from this. In the case $d\ge 2$ and $\alpha>\frac{d}{2}$, the probability of $\{A(t)\ge t(\E[z(0)]+\epsilon)\}\cup\{A(t)\le t(\E[z(0)]-\epsilon)\}$ decays stretched exponentially and the result can be refined as stated. 

The case $d\ge 3$ with $\alpha<\frac{d}{2}$ is a bit more involved since there is an unusual power law regime in Theorem~\ref{thm:RCM-power}. 
We divide the integral on $[n^{c(\epsilon)},\infty)$ into two parts and use Theorem~\ref{thm:RCM-power} with a discretization argument similar to~\eqref{discretization} to find
\begin{equation}
\label{eq:Green1}
 \begin{split}
\int_{n^{c(\epsilon)}}^{n^{2(1\wedge \alpha)-\epsilon}} P^\omega_0(X_t=n\mathbf{e}_1) \dd t
&=\int_{n^{4\alpha/d-2\epsilon}}^{n^{2(1\wedge \alpha)-\epsilon}} n^{-1-2\alpha}t^{1-\frac{d}{2}+o(1)} \dd t\\
&=
\begin{cases}
n^{-1-2\alpha+(4-d)(1\wedge \alpha)+o(1)}, &d\in \{3, 4\}, \\
n^{-1-\frac{4\alpha}{d}(d-2)+o(1)}, &d\ge 5,\\
\end{cases}
\end{split}
\end{equation}
\begin{equation}
\label{eq:Green3}
 \begin{split}
\int_{n^{2(1\wedge \alpha)-\epsilon}}^\infty P^\omega_0(X_t=n\mathbf{e}_1) \dd t
&=\int_{n^{2(1\wedge \alpha)-\epsilon}}^\infty t^{-\frac{1}{2}(\frac{1}{\alpha}\vee 1)-\frac{d}{2}+o(1)} \dd t\\
&=n^{-1-(1\wedge \alpha)(d-2)+o(1)}
 \end{split}
\end{equation}
as $n\to\infty$ followed by $\epsilon\to 0$. Note that the asymptotics in~\eqref{eq:Green1} are determined by the upper limit of the integral when $d=3\text{ or }4$, and by the lower limit when $d\ge 5$. When $d \ge 5$ and $\alpha<\frac{d}{4}$, the leading term is~\eqref{eq:Green1}. Otherwise, it is~\eqref{eq:Green3}. Combining these estimates, we obtain the desired asymptotics.
\end{proof}
\begin{remark}
\label{rem:Green-off}
\begin{enumerate}[leftmargin=20pt]
 \item 
If we change the endpoint $n\mathbf{e}_1$ to $n(\mathbf{e}_1+\mathbf{e})$ with a non-zero $\mathbf{e}\in\R^{1+d}$ perpendicular to $\mathbf{e}_1$, then we have an upper bound
\begin{equation}
\begin{split}
g^\omega(0,n(\mathbf{e}_1+\mathbf{e})) &\le n^{-s(d,\alpha)-d+2+o(1)}\\
& = \bar{d}(0,n(\mathbf{e}_1+\mathbf{e}))^{2-d_{\text{s}}+o(1)}
\end{split}
\end{equation}
as $n\to \infty$, where the exponent in the first line is larger than those in Theorem~\ref{thm:Green} when $\alpha<1\vee \frac{d}{4}$. Therefore in this case, the decay rate of the Green function is singular in $\mathbf{e}_1$-direction and the elliptic Harnack inequality fails. Indeed, it was~\eqref{eq:Green1} that caused the anomalous behavior but when $\mathbf{e}\neq 0$, we have
\begin{equation}
\int_{n^{c(\epsilon)}}^{n^{2-\epsilon}} P^\omega_0(X_t=n(\mathbf{e}_1+\mathbf{e})) \dd t
\le \int_{n^{4\alpha/d-2\epsilon}}^{n^{2-\epsilon}} P_0(S^2(t)=n\mathbf{e}) \dd t, 
\end{equation}
which decays stretched exponentially in $n$ by Lemma~\ref{lem:RWHK}. For the other part of the integral, recall that $p_{A(t)}(0,n) \le cA(t)^{-\frac12}$ by Lemma~\ref{lem:RWHK} and that we may assume $A(t) \ge t^{s(d,\alpha)-\epsilon}$ by Proposition~\ref{prop:RWRSlower}. 
Then we have 
\begin{equation}
\begin{split}
& \int_{n^{2-\epsilon}}^\infty P^\omega_0(X_t=n(\mathbf{e}_1+\mathbf{e})) \dd t\\
&\quad \sim \int_{n^{2-\epsilon}}^\infty E_0\left[p_{A(t)}(0,n); A(t) \ge t^{s(d,\alpha)-\epsilon}, S^2(t)=n\mathbf{e}\right] \dd t\\
&\quad \le c \int_{n^{2-\epsilon}}^\infty t^{-\frac12(s(d,\alpha)-\epsilon)} P_0\left(S^2(t)=n\mathbf{e}\right) \dd t\\
&\quad \le n^{-s(d,\alpha)-d+2+c\epsilon}
\end{split}
\label{eq:Green-off-upper}
\end{equation}
for any $\epsilon>0$ as $n\to\infty$. 
 \item Suppose $d\ge 3$ and $1\le \alpha<\frac{d}{2}$. If we fix $\epsilon>0$ sufficiently small, then just as in~\eqref{eq:LCLTfail}, we can find a sequence $\{x_2(n)\}_{n\in\N}\subset \{0\}\times \Z^d$ such that $n\le |x_2(n)|\le 2 n$ and 
\begin{equation}
 P_0^\omega(X_t=n\mathbf{e}_1+x_2(n)) \le n^{-1-d-2\epsilon}
\text{ for all }t\in[n^{2-\epsilon},n^{2+\epsilon}].
\end{equation}
Using this bound in~\eqref{eq:Green-off-upper} for $t\in[n^{2-\epsilon},n^{2+\epsilon}]$ instead of Lemma~\ref{lem:RWHK}, we obtain
\begin{equation}
 g^\omega(0,n\mathbf{e}_1+x_2(n)) \le c n^{1-d-\epsilon}.
\end{equation}
Since this decays faster than $g^\omega(0,n\mathbf{e}_1)$, it follows that the elliptic Harnack inequality fails in this case. 
 \item For the constant speed random walk $Y_t = (S^1_{A^2(B^{-1}(t))}, S^2_{B^{-1}(t)})\in \Z^{d_1+d_2}$ in~\eqref{eq:CSRW} and any $x_1\in \R^{d_1}\setminus\{0\}$, we have
\begin{equation}
\begin{split}
\tilde{g}^\omega(0,n(x_1,0))
&=\int_0^\infty P_0^\omega(Y_t=n(x_1,0))\dd t\\
&= \int_0^\infty E_0\bigl[p_{A(B^{-1}(t))}(0,nx_1);
S_{B^{-1}(t)}=0\bigr]\dd t.
\end{split}
\label{eq:Green-CSRW}
\end{equation}
It is simple to check that $A(B^{-1}(t))\asymp t$, and hence $p_{A(B^{-1}(t))}(0,nx_1)$ satisfies the same upper and lower bounds as in Lemma~\ref{lem:RWHK}. If we further assume that the following limit exists
\begin{equation}
d_{\text{s}}^{\text{BTM}}=-2\lim_{t\to\infty}\frac{\log P_0(S_{B^{-1}(t)}=0)}{\log t},
\label{eq:d_s-BTM}
\end{equation}
then by substituting the bounds in Lemma~\ref{lem:RWHK} and~\eqref{eq:d_s-BTM} into~\eqref{eq:Green-CSRW}, we can deduce that 
\begin{equation}
\tilde{g}^\omega(0,n(x_1,0))
=n^{2-d_1-d_{\text{s}}^{\text{BTM}}+o(1)}
\end{equation}
as $n\to \infty$. Since $\tilde{g}^\omega(x,y)=g^\omega(x,y)\omega(y)$ obeys the same asymptotics as in~\eqref{eq:Green-MD} on the other hand, it follows that
\begin{equation}
 d_{\text{s}}^{\text{BTM}}=
\begin{cases}
1\wedge \frac{2}{\alpha+1}, & d_2=1, \text{ and $\alpha<1$ if $d_1=1$},\\[5pt]
2+\bigl(1\wedge \alpha \wedge \frac{4\alpha}{d_2}\bigr)(d_2-2), & d_2\ge 2.
\end{cases}
\end{equation}
By using Lemma~\ref{lem:RWHK} again, we have 
\begin{equation}
\begin{split}
P_0^\omega(Y_t=n(x_1,0))
&= E_0\bigl[p_{A(B^{-1}(t))}(0,nx_1);S_{B^{-1}(t)}=0\bigr]\\
&= t^{-\frac{d_1}{2}-\frac{d_{\text{s}}^{\text{BTM}}}{2}+o(1)}
\end{split}
\end{equation}
and this yields the value of spectral dimension stated in~\eqref{eq:d_s-CSRW} (under the assumption that the limit~\eqref{eq:d_s-BTM} exists).
\end{enumerate}
\end{remark}
\section*{Acknowledgments}
The authors are grateful to Amir Dembo for suggesting to study the on-diagonal heat kernel behavior, Amine Asselah and Fabienne Castell for a discussion on the rate function in~\cite{AC03b}, and Omar Boukhadra for useful comments on an early version of this paper. The second author was partially supported by JSPS KAKENHI Grant Number 16K05200 and ISHIZUE 2019 of Kyoto University Research Development Program. 


\begin{thebibliography}{10}

\bibitem{ADS16b}
S.~Andres, J.-D.~Deuschel, and M.~Slowik.
\newblock Harnack inequalities on weighted graphs and some applications to the
  random conductance model.
\newblock {\em Probab. Theory Related Fields}, 164(3-4):931--977, 2016.

\bibitem{ADS19}
S.~Andres, J.-D.~Deuschel, and M.~Slowik.
\newblock Heat kernel estimates and intrinsic metric for random walks with general speed measure under degenerate conductances. 
\newblock {\em Electron. Commun. Probab.}, 24, paper no.~5, 17 pp., 2019.

\bibitem{AC03b}
A.~Asselah and F.~Castell.
\newblock Large deviations for {B}rownian motion in a random scenery.
\newblock {\em Probab. Theory Related Fields}, 126(4):497--527, 2003.

\bibitem{BAC07}
G.~Ben Arous,  and J.~\v{C}ern\'{y}.
\newblock Scaling limit for trap models on {$\Bbb Z^d$}.
\newblock {\em Ann. Probab.}, 35(6):2356--2384, 2007.

\bibitem{BD10}
M.~Barlow and J.-D.~Deuschel. 
\newblock Invariance principle for the random conductance model with unbounded conductances. 
\newblock {\em Ann. Probab.}, 38(1):234--276, 2010.

\bibitem{BS19}
P.~Bella and M.~Sch\"{a}ffner.
\newblock Quenched invariance principle for random walks among random degenerate conductances. 
\newblock to appear in \emph{Ann. Probab.}, arXiv:1902.05793.

\bibitem{BAR00}
G.~Ben Arous and A.~F.~Ram\'irez.
\newblock Asymptotic survival probabilities in the random saturation process.
\newblock {\em Ann. Probab.}, 28(4):1470--1527, 2000.

\bibitem{Bis11}
M.~Biskup.
\newblock Recent progress on the random conductance model.
\newblock {\em Probab. Surveys}, 8:294--373, 2011.

\bibitem{BK01b}
M.~Biskup and W.~K{\"o}nig.
\newblock Long-time tails in the parabolic {A}nderson model with bounded
  potential.
\newblock {\em Ann. Probab.}, 29(2):636--682, 2001.

\bibitem{Bor79a}
A.~N.~Borodin.
\newblock A limit theorem for sums of independent random variables defined on a
  recurrent random walk.
\newblock {\em Dokl. Akad. Nauk SSSR}, 246(4):786--787, 1979.

\bibitem{Bor79b}
A.~N.~Borodin.
\newblock Limit theorems for sums of independent random variables defined on a
  transient random walk.
\newblock {\em Zap. Nauchn. Sem. Leningrad. Otdel. Mat. Inst. Steklov. (LOMI)},
  85:17--29, 237, 244, 1979.
\newblock Investigations in the theory of probability distributions, IV.


\bibitem{CGPP13}
F.~Castell, N.~Guillotin-Plantard, and F.~P\`ene.
\newblock Limit theorems for one and two-dimensional random walks in random
  scenery.
\newblock {\em Ann. Inst. Henri Poincar\'{e} Probab. Stat.}, 49(2):506--528,
  2013.

\bibitem{Cha12}
A.~Chakrabarty. 
\newblock Effect of truncation on large deviations for heavy-tailed random vectors. 
\newblock {\em Stochastic Process. Appl.}, 122(2):623--653, 2012. 

\bibitem{CHK19}
D.~A.~Croydon, B.~M.~Hambly, and T.~Kumagai.
\newblock Heat kernel estimates for {FIN} processes associated with resistance forms. 
\newblock {\em Stochastic Process. Appl.}, 129(9), 2991--3017, 2019. 

\bibitem{DD05}
T.~Delmotte and J.-D.~Deuschel.
\newblock On estimating the derivatives of symmetric diffusions in stationary
  random environment, with applications to {$\nabla\phi$} interface model.
\newblock {\em Probab. Theory Related Fields}, 133(3):358--390, 2005.

\bibitem{DF18}
J.-D.~Deuschel and R.~Fukushima.
\newblock Quenched tail estimate for the random walk in random scenery and in
  random layered conductance.
\newblock {\em Stochastic Process. Appl.}, 129(1):102--128,
  2019.

\bibitem{DV79}
M.~D.~Donsker and S.~R.~S.~Varadhan.
\newblock On the number of distinct sites visited by a random walk.
\newblock {\em Comm. Pure Appl. Math.}, 32(6):721--747, 1979.

\bibitem{DE51}
A.~Dvoretzky and P.~Erd\H{o}s.
\newblock Some problems on random walk in space.
\newblock In {\em Proceedings of the {S}econd {B}erkeley {S}ymposium on
  {M}athematical {S}tatistics and {P}robability, 1950}, pages 353--367, 1951.

\bibitem{Fel68}
W.~Feller.
\newblock {\em An introduction to probability theory and its applications.
  {V}ol. {I}}.
\newblock Third edition. John Wiley \& Sons, Inc., New York-London-Sydney,
  1968.

\bibitem{Fuk09a}
R.~Fukushima.
\newblock Brownian survival and {L}ifshitz tail in perturbed lattice disorder.
\newblock {\em J. Funct. Anal.}, 256(9):2867--2893, 2009.

\bibitem{Fuk09b}
R.~Fukushima.
\newblock From the {L}ifshitz tail to the quenched survival asymptotics in the
  trapping problem.
\newblock {\em Electron. Commun. Probab.}, 14:435--446, 2009.

\bibitem{GPPdS14}
N.~Guillotin-Plantard, J.~Poisat, and R.~S.~dos Santos.
\newblock A quenched functional central limit theorem for planar random walks
  in random sceneries.
\newblock {\em Electron. Commun. Probab.}, 19, paper no.~3, 9 pp., 2014.

\bibitem{KS79}
H.~Kesten and F.~Spitzer.
\newblock A limit theorem related to a new class of self-similar processes.
\newblock {\em Z. Wahrsch. Verw. Gebiete}, 50(1):5--25, 1979.

\bibitem{Kon16}
W.~K\"onig.
\newblock {\em The parabolic {A}nderson model}.
\newblock Pathways in Mathematics. Birkh\"auser/Springer, [Cham], 2016.
\newblock Random walk in random potential.

\bibitem{LL10}
G.~F.~Lawler and V.~Limic.
\newblock {\em Random walk: a modern introduction}, volume 123 of {\em
  Cambridge Studies in Advanced Mathematics}.
\newblock Cambridge University Press, Cambridge, 2010.

%

\bibitem{Szn93c}
A.-S.~Sznitman.
\newblock Brownian asymptotics in a {P}oissonian environment.
\newblock {\em Probab. Theory Related Fields}, 95(2):155--174, 1993.

\bibitem{Szn98}
A.-S.~Sznitman.
\newblock {\em Brownian motion, obstacles and random media}.
\newblock Springer Monographs in Mathematics. Springer-Verlag, Berlin, 1998.

\end{thebibliography}
\newcommand{\noop}[1]{}

\end{document}